\documentclass[11pt]{amsart}

\usepackage{amsmath,amssymb,amsthm,mathrsfs}
\usepackage[a4paper,margin=1in]{geometry}
\usepackage{enumitem}
\usepackage[colorlinks=true,linkcolor=black,citecolor=black,urlcolor=black]{hyperref}

\theoremstyle{plain}
\newtheorem{theorem}{Theorem}[section]
\newtheorem{proposition}[theorem]{Proposition}
\newtheorem{lemma}[theorem]{Lemma}
\newtheorem{corollary}[theorem]{Corollary}
\theoremstyle{definition}
\newtheorem{definition}[theorem]{Definition}
\theoremstyle{remark}
\newtheorem{remark}[theorem]{Remark}

\newcommand{\T}{\mathbb T}
\newcommand{\R}{\mathbb R}
\newcommand{\Z}{\mathbb Z}
\newcommand{\supp}{\operatorname{supp}}
\newcommand{\diag}{\operatorname{diag}}
\newcommand{\Id}{\mathrm{Id}}
\title[Energy discontinuity of the Lyapunov exponent]{Energy Discontinuity of the Lyapunov Exponent
	for Quasiperiodic Schr\"odinger Operators
	with Gevrey Potentials}

\author{Jiahao Xu}
\address{Beijing International Center for Mathematical Research, Peking University,
	Yiheyuan Road, Beijing 100871, China}
\address{School of Mathematics, Nanjing University,
	Hankou Road, Nanjing 210093, China}
\email{xjhao512@outlook.com}
\thanks{These authors contributed equally to this work.}

\author{Shuo Zhang}
\address{School of Mathematics, Nanjing University,
	Hankou Road, Nanjing 210093, China}
\email{shuozhang@smail.nju.edu.cn}
\thanks{Corresponding author: Shuo Zhang.}

\begin{document}
	
	\begin{abstract}
		We construct quasiperiodic Schr\"odinger operators with Gevrey
		potentials for which the Lyapunov exponent is discontinuous in the
		spectral energy. More precisely, let \(\alpha\in\R\setminus\mathbb Q\)
		satisfy the Diophantine condition
		\(\|k\alpha\|_{\T}\geq\gamma |k|^{-\tau}\) for all
		\(k\in\Z\setminus\{0\}\), with \(\gamma>0\) and \(\tau\geq1\).
		For every Gevrey index \(\mathfrak s>1+\tau\) and every sufficiently
		large coupling \(\lambda\), we construct a real-valued potential
		\(v_\lambda\in G^{\mathfrak s}(\T,\R)\) and a sequence of energies
		\(E_j\downarrow0\) such that
		\[
		L_{v_\lambda}(0)\geq(1-\varepsilon_\lambda)\log\lambda,
		\qquad
		\sup_{j\geq1} L_{v_\lambda}(E_j)\leq(1-c_*)\log\lambda,
		\]
		where \(L_{v_\lambda}(E)\) denotes the Lyapunov exponent of the
		Schr\"odinger operator with potential \(\lambda v_\lambda\),
		\(c_*>0\) is an absolute constant, and \(\varepsilon_\lambda\to0\)
		as \(\lambda\to\infty\). For bounded-type frequencies (\(\tau=1\)),
		the regularity range \(\mathfrak s>2\) is sharp up to the endpoint
		\(\mathfrak s=2\), since the Lyapunov exponent is known to be
		continuous in energy for \(1<\mathfrak s<2\).
	\end{abstract}
	
	\maketitle
	
	\section{Introduction and main result}
	\label{sec1}
	
	Let \(\alpha\in\R\setminus\mathbb Q\) be an irrational frequency,
	\(x\in\T=\R/\Z\) a phase, \(v\colon\T\to\R\) a potential, and
	\(\lambda>1\) a coupling constant. We consider the one-dimensional
	quasiperiodic Schr\"odinger operator \(H_{\alpha,\lambda v,x}\) acting
	on \(\ell^2(\Z)\) by
	\[
	(H_{\alpha,\lambda v,x}u)_m
	=u_{m+1}+u_{m-1}+\lambda v(x+m\alpha)u_m .
	\]
	Write \(Tx=x+\alpha\pmod1\) for the irrational rotation on \(\T\) and
	\(|x|:=\operatorname{dist}_{\T}(x,0)\) for the distance to the origin on
	the torus. Throughout the paper, \(\log^+t:=\max\{\log t,0\}\),
	\(E_{ij}\) denotes the standard elementary matrix with \(1\) in the
	\((i,j)\)-entry and \(0\) elsewhere, and matrix norms \(\|\cdot\|\)
	are operator norms induced by the Euclidean norm on \(\R^2\).
	
	At energy \(E\in\R\), the difference equation \(H_{\alpha,\lambda v,x}u=Eu\)
	is governed by the Schr\"odinger cocycle \((\alpha,A_{v,E})\) on
	\(\T\times\R^2\), generated by the one-step transfer matrix
	\begin{equation*}
		A_{v,E}(x)=
		\begin{pmatrix}E-\lambda v(x)&-1\\1&0\end{pmatrix}\in\mathrm{SL}(2,\R).
	\end{equation*}
	Its forward iterates are given by
	\begin{equation*}
		A_{v,E}^{(m)}(x)
		=A_{v,E}(T^{m-1}x)\cdots A_{v,E}(x),
		\qquad m\geq1,
	\end{equation*}
	with \(A_{v,E}^{(0)}=\Id\). By Kingman's subadditive ergodic theorem,
	the (upper) Lyapunov exponent is given by
	\begin{equation}\label{eqLE}
		L_v(E)=\lim_{m\to\infty}\frac1m
		\int_{\T}\log\|A_{v,E}^{(m)}(x)\|\,dx
		=\inf_{m\geq1}\frac1m
		\int_{\T}\log\|A_{v,E}^{(m)}(x)\|\,dx .
	\end{equation}
	
	A central problem in the spectral theory of quasiperiodic Schr\"odinger
	operators is to understand how the regularity of the sampling function
	\(v\) and the arithmetic of the frequency \(\alpha\) govern the
	regularity of \(L_v(E)\). While \(L_v(E)\) is always continuous in \(E\)
	when \(v\) is real analytic, the behavior in ultradifferentiable and
	smooth classes below analyticity is more delicate. We ask whether \(E\mapsto L_v(E)\) must remain continuous when a
	Gevrey potential \(v\) is held fixed. For general background on
	one-dimensional ergodic Schr\"odinger operators and their Lyapunov
	exponents, we refer to the monographs of Damanik--Fillman
	\cite{DamanikFillmanI,DamanikFillmanII}.
	
	\begin{definition}
		\label{def1}
		For an interval \(I\subset\R\), a Gevrey exponent \(\mathfrak s>1\), and
		a radius parameter \(K\geq1\), we define the Gevrey norm
		\begin{equation}\label{def1a}
			|F|_{\mathfrak s,K;I}
			:=\frac{4\pi^2}{3}\sup_{r\geq0}
			\frac{(1+r)^2}{K^r(r!)^{\mathfrak s}}
			\|\partial_x^rF\|_{C^0(I)}.
		\end{equation}
		Let \(G^{\mathfrak s,K}(I)\) denote the Banach space of smooth
		functions on \(I\) with \(|F|_{\mathfrak s,K;I}<\infty\), and set
		\(G^{\mathfrak s}(I)=\bigcup_{K\geq1}G^{\mathfrak s,K}(I)\).
		For periodic functions on \(\T\), the same norm is taken over one
		period. The normalizing factor \(4\pi^2/3\) is chosen so that
		\(G^{\mathfrak s,K}(I)\) forms a Banach algebra under pointwise
		multiplication (see Lemma~\ref{lem4}).
	\end{definition}
	
	An exponent \(t\geq1\) is \emph{admissible for \(\alpha\)} if there exists
	a constant \(\gamma_t>0\) such that
	\[
	\|k\alpha\|_{\T}\geq\gamma_t |k|^{-t},
	\qquad k\in\Z\setminus\{0\}.
	\]
	In terms of the principal convergents \(p_n/q_n\) of \(\alpha\),
	admissibility of \(t\) is equivalent to the denominator growth bound
	\(q_{n+1}\leq C(\gamma_t,t)q_n^t\) for all \(n\geq1\).
	We distinguish a specific admissible exponent \(t\) from
	the intrinsic Diophantine exponent
	\begin{equation}\label{eq1}
		\tau_*(\alpha)
		:=\inf\bigl\{t\geq1:t\text{ is admissible for }\alpha\bigr\}
		=\limsup_{n\to\infty}
		\frac{\log q_{n+1}}{\log q_n},
	\end{equation}
	where we set \(\tau_*(\alpha)=+\infty\) if \(\alpha\) is Liouville.
	The infimum in \eqref{eq1} need not be attained
	among admissible exponents. Throughout our
	construction, we fix an admissible exponent \(\tau\geq1\) and a
	constant \(\gamma>0\) such that
	\begin{equation}\label{eq2}
		\|k\alpha\|_{\T}\geq\gamma |k|^{-\tau},
		\qquad k\in\Z\setminus\{0\}.
	\end{equation}
	
	Our main theorem establishes energy discontinuity of the Lyapunov
	exponent for a fixed Gevrey potential in the regularity range
	\(\mathfrak s>1+\tau\).
	
	\begin{theorem}
		\label{thm1}
		Let \(\alpha\in\R\setminus\mathbb Q\) satisfy the Diophantine condition
		\eqref{eq2} with parameters \(\tau\geq1\) and \(\gamma>0\), and
		let \(\mathfrak s>1+\tau\). For every sufficiently large coupling
		\(\lambda\), there exist a potential
		\(v_{\lambda,\alpha,\mathfrak s}\in G^{\mathfrak s}(\T,\R)\) and a
		sequence of energies \(E_j\downarrow0\) such that
		\[
		L_{v_{\lambda,\alpha,\mathfrak s}}(0)
		\geq(1-\varepsilon_\lambda)\log\lambda,
		\qquad
		\sup_{j\geq1}L_{v_{\lambda,\alpha,\mathfrak s}}(E_j)
		\leq(1-c_*)\log\lambda,
		\]
		where \(\varepsilon_\lambda\to0\) as \(\lambda\to\infty\) and
		\(c_*>0\) is an absolute constant. In particular, the energy map
		\(E\mapsto L_{v_{\lambda,\alpha,\mathfrak s}}(E)\) is discontinuous at
		\(E=0\). When \(\alpha\) is of bounded type, this conclusion holds for
		every \(\mathfrak s>2\).
	\end{theorem}
	
	Theorem~\ref{thm1} resolves in the negative the fixed-potential
	energy-continuity question discussed by Jitomirskaya in her ICM address
	\cite[Section~2, p.~1096]{JitomirskayaICM2022}, and provides the
	fixed-potential counterexamples sought in \cite[Problem~4]{LTY2025};
	see also You \cite{YouProblems2026} for a recent survey of open
	problems in the theory of quasiperiodic Schr\"odinger operators.
	
	Every \(t>\tau_*(\alpha)\) is admissible for \(\alpha\). Applying
	Theorem~\ref{thm1} with any
	\(\tau\in(\tau_*(\alpha),\mathfrak s-1)\) therefore yields fixed-potential
	energy discontinuity at large coupling whenever
	\begin{equation*}
		\mathfrak s>1+\tau_*(\alpha).
	\end{equation*}
	In particular, \(\tau_*(\alpha)=1\) whenever \(q_{n+1}\le Cq_n(\log q_n)^A\)
	for some constants \(C,A>0\) and all sufficiently large \(n\), a class
	that includes all bounded-type frequencies as well as a set of full
	Lebesgue measure (for \(A>1\)).
	
	\subsection*{Context and related results}
	
	For real-analytic potentials, Bourgain--Jitomirskaya
	\cite{BourgainJitomirskaya} proved that the Lyapunov exponent of a
	one-frequency Schr\"odinger operator is jointly continuous in energy and
	irrational frequency. Subsequent work extended continuity to general
	analytic cocycles, including singular and higher-rank models
	\cite{JitomirskayaKosloverSchulteis,JitomirskayaMarx,
		AvilaJitomirskayaSadel,DuarteKlein}, while quantitative moduli of
	continuity under arithmetic and positivity assumptions were obtained in
	\cite{GoldsteinSchlag,YouZhang}.
	
	Below analyticity, continuity for general one-frequency Gevrey cocycles
	over Diophantine rotations is known in the regime
	\begin{equation}\label{eq3}
		1<\mathfrak s<2.
	\end{equation}
	For strongly Diophantine frequencies, Klein \cite{Klein2005} proved
	energy continuity in the range \eqref{eq3}
	under a uniform positivity assumption on the Lyapunov exponent, as well
	as for arbitrary Gevrey index \(\mathfrak s>1\) at large coupling under
	a transversality condition. Unconditional continuity for general
	\(\mathrm{SL}(2,\R)\)-valued \(G^{\mathfrak s}\) cocycles in the range
	\eqref{eq3} was established by
	Cheng--Ge--You--Zhou \cite[Theorem~6.3]{ChengGeYouZhou}. Wang
	\cite[Theorem~1.1]{XueyinWang2026} extended this result to frequencies
	in a subexponential Brjuno class satisfying
	\(\log q_{n+1}\leq C(\alpha,\eta)q_n^\eta\) with \(0<\eta<2-\mathfrak s\).
	Since every Diophantine frequency satisfies this growth condition for
	all \(\eta>0\) and \(A_{v,E}-A_{v,E_0}=(E-E_0)E_{11}\), energy
	continuity holds for every fixed \(v\in G^{\mathfrak s}(\T,\R)\) when
	\(1<\mathfrak s<2\). See also Powell \cite{PowellGevrey2025} for direct
	Gevrey large-deviation theorems and continuity consequences in an
	arithmetic-dependent subrange.
	
	The restriction \(\mathfrak s<2\) in the Fourier approximation method
	comes from the relation between truncation error and complex strip width.
	A cutoff at frequency \(N^{b\mathfrak s}\) gives an error
	\(O(e^{-cN^b})\) and a trigonometric polynomial \(v_N\) with
	\[
	\deg v_N\leq C N^{b\mathfrak s},\qquad
	\text{strip width}\asymp N^{-b(\mathfrak s-1)}.
	\]
	The large-deviation argument requires \(b>1\) and
	\(b(\mathfrak s-1)<1\), which are compatible precisely when
	\(\mathfrak s<2\). Restricting parameter variation to the energy line
	does not change this loss in the phase variable.
	
	Wang--You \cite{WangYou,WangYou2018}
	constructed examples of discontinuity of the Lyapunov exponent in the
	\(C^\infty\) cocycle topology. For bounded-type frequencies,
	Ge--Wang--You--Zhao \cite{GWYZ} refined this mechanism to obtain
	discontinuity in the \(G^{\mathfrak s}\) topology for \(\mathfrak s>2\).
	Liang--Tao--You \cite{LTY2025} incorporated the Diophantine
	exponent \(\tau\) and constructed discontinuity examples in
	\(G^{\mathfrak s}\) for \(\mathfrak s>1+\tau\) under \eqref{eq2}.
	In these constructions, the cocycle or potential varies at a fixed
	energy. In \cite{LTY2025}, the authors explicitly
	distinguished continuity in the cocycle topology (Problem~2) from energy
	continuity for a fixed potential (Problem~4). Theorem~\ref{thm1}
	reaches the same arithmetic range \(\mathfrak s>1+\tau\) while
	holding the potential \(v\) fixed and varying only the spectral energy
	\(E\). (In the topologically distinct setting of limit-periodic
	potentials, fixed-potential energy discontinuity was obtained by
	Damanik--Gan--Kr\"uger \cite{GanKrueger}.)
	
	When additional geometric structure is imposed on the potential,
	continuity can hold beyond \(\mathfrak s=2\): besides Klein's transverse
	class, large-coupling \(C^2\) cosine-type potentials with two
	nondegenerate critical points admit continuity and quantitative regularity
	results under the respective arithmetic assumptions
	\cite{WangZhang,LiangFu,LiangWangYou,XuGeWang}. For
	arbitrary fixed Gevrey potentials over a Diophantine rotation, the
	preceding results give the following ranges:
	\begin{equation*}
		\begin{array}{c|c}
			1<\mathfrak s<2
			& \text{continuity for every fixed }v\in G^{\mathfrak s}(\T,\R),\\[2mm]
			2\leq\mathfrak s\leq1+\tau_*(\alpha)
			& \text{not settled by the cited results},\\[2mm]
			\mathfrak s>1+\tau_*(\alpha)
			& \text{discontinuity examples at large coupling (Theorem~\ref{thm1}).}
		\end{array}
	\end{equation*}
	When \(\tau_*(\alpha)=1\), only the endpoint \(\mathfrak s=2\) remains
	uncovered by these results. By \eqref{eqLE}, \(E\mapsto L_v(E)\) is upper
	semicontinuous and nonnegative, so it is automatically continuous
	wherever \(L_v(E)=0\) as well as on the uniformly hyperbolic resolvent
	set. Any discontinuity must therefore occur at a spectral energy with
	strictly positive Lyapunov exponent.
	
	\subsection*{Strategy of the proof}
	
	Passing from potential-variation discontinuity \cite{WangYou,GWYZ,LTY2025}
	to fixed-potential energy discontinuity introduces an additional
	difficulty: varying \(E\) changes the transfer matrix at every orbit
	point, since \(A_{v,E}(x)-A_{v,0}(x)=E\,E_{11}\). The open return-gap
	products \(B_{I_n}^\pm(E,x;v)\), and hence their singular values and
	directions, therefore depend on \(E\).
	
	The construction uses the following return-block estimates and perturbations:
	
	\smallskip
	\noindent\textbf{1. First-return factorization and affine decoupling.}
	In Section~\ref{sec2}, we factor the return product centered
	at a visit \(x\in I_n\) using the singular value decompositions of the
	adjacent open gap blocks \(B_{I_n}^\pm(E,x;v)\). The coupling between
	the incoming and outgoing expanding singular directions is governed by
	the \((1,1)\)-entry of the rotated connecting matrix
	\(C_{I_n}(E,x;v)=R_{\phi_{I_n}^+}A_{v,E}(x)R_{\theta_{I_n}^-}\).
	Dividing this entry by \(\cos\phi_{I_n}^+\cos\theta_{I_n}^-\) yields the
	normalized connecting function
	\[
	g_{I_n}(E,x;v)=E-\lambda v(x)-\tan\phi_{I_n}^+(E,x)-\tan\theta_{I_n}^-(E,x).
	\]
	By the first-return property, the open gap blocks \(B_{I_n}^\pm(E,x;v)\)
	sample the potential only outside \(I_n\). Hence any perturbation
	\(\delta_n\) compactly supported in \(I_n\) leaves \(B_{I_n}^\pm(E,x;v)\)
	and their singular directions invariant for \emph{all} energies \(E\),
	shifting \(g_{I_n}\) affinely by
	\(g_{I_n}(E,x;v+\delta_n)=g_{I_n}(E,x;v)-\lambda\delta_n(x)\). In
	particular, \(\partial_E g_{I_n}\) is unchanged by \(\delta_n\).
	
	\smallskip
	\noindent\textbf{2. Logarithmic cutoff interpolation and multi-energy resonance.}
	Starting from a flat profile \(h(x)=\exp(-|\sin\pi x|^{-\nu})\) at a
	large initial scale \(n_0(\lambda)\), we inductively construct
	corrections \(\delta_n\) supported in nested subintervals
	\(\widetilde I_n\Subset\mathcal K_{n-1}\subset I_{n-1}\). At scale \(n\),
	\(\delta_n\) serves two purposes:
	\begin{itemize}
		\item On a central core \(\mathcal K_n\subset I_n\), it resets the
		zero-energy connecting function to the model profile,
		\(g_{I_n}(0,x;v_n)=-\lambda h(x)\), preparing the next scale;
		\item On two symmetric side plateaus \(P_{n,\pm}\subset I_n\setminus\mathcal K_n\),
		it enforces an exact resonance \(g_{I_n}(E_n,x;v_n)=0\) at the selected
		energy \(E_n=\lambda h(\vartheta_d\rho_n)\), annihilating the
		expanding-to-expanding entry \((C_{I_n}(E_n,x;v_n))_{11}\) across a
		set of phases of measure proportional to \(|I_n|\).
	\end{itemize}
	To keep the zero-energy Lyapunov exponent
	\(L_{v_\infty}(0)\) close to \(\log\lambda\), \(g_{I_n}(0,x;v_n)\) must
	remain strictly negative (bounded below in magnitude by
	\(c\lambda e^{-\beta|x|^{-\nu}}\)) across all cutoff transition zones.
	On the side buffers, the zero-energy profile \(A_n(x)=\lambda h(x)\)
	and the target energy increment
	\(B_{n,\pm}(x)=\int_0^{E_n}\partial_E g_{I_n}(u,x;v_{n-1})\,du\approx E_n\)
	differ by a large exponential factor. To control the reciprocal
	throughout the transition, Section~\ref{sec6} uses
	the logarithmic interpolant
	\(\mathcal H_{n,\pm}=A_n^{1-\chi_{n,\pm}}B_{n,\pm}^{\chi_{n,\pm}}\).
	The resulting correction preserves the negative sign of the zero-energy
	connecting function and, for \(0\leq E\leq E_{n+1}\), bounds the Gevrey
	norm of the reciprocal of the corrected connecting function on these
	buffers by
	\(\exp(C\rho_n^{-\nu}+C\rho_n^{-\mathfrak b})\).
	
	\smallskip
	\noindent\textbf{3. First-order jet Gevrey calculus and energy transversality.}
	Prescribing the zero \(g_{I_n}(E_n,x;v_n)=0\) via the fundamental
	theorem of calculus requires uniform energy transversality
	\(\partial_E g_{I_n}\approx 1\) on \([0,E_n]\), together with spatial
	Gevrey bounds on \(\partial_E\phi_{I_n}^+\) and
	\(\partial_E\theta_{I_n}^-\). In Section~\ref{sec4}, we
	embed each function and its first energy derivative into the dual-number
	Gevrey Banach algebra
	\(\mathcal G_{\mathfrak s,K}^{(1)}(I)=G^{\mathfrak s,K}(I)\oplus\epsilon G^{\mathfrak s,K}(I)\)
	(\(\epsilon^2=0\)), controlling each return block and its first energy
	derivative simultaneously without an additional loss of Gevrey
	regularity. Combined with the continued-fraction return decomposition of
	Section~\ref{sec5}, this allows us to close the
	multi-scale induction in Sections~\ref{sec7}
	and~\ref{sec8} at a \emph{fixed} Gevrey radius
	\(K_\lambda\), and to derive the two Lyapunov exponent bounds
	in Section~\ref{sec10}.

	\section{First-return factorization and mixed Gevrey calculus}
	\label{sec2}
	
	Throughout the paper, we write \(e_1=(1,0)^{\mathsf T}\). The prefactor
	\(4\pi^2/3\) in Definition~\ref{def1} ensures that
	\(G^{\mathfrak s,K}(I)\) is a Banach algebra under pointwise
	multiplication:
	\begin{equation}\label{eq4}
		|FG|_{\mathfrak s,K;I}
		\leq |F|_{\mathfrak s,K;I}|G|_{\mathfrak s,K;I}.
	\end{equation}
	
	\begin{remark}[Convention on constants]
		\label{rem1}
		Throughout the paper, \(C,c>0\) denote generic positive constants that
		may vary from line to line. Unless stated otherwise, they may depend on
		the fixed Diophantine parameters \((\gamma,\tau)\) and the Gevrey
		regularity exponents, but are always independent of the coupling
		\(\lambda\), the induction scale \(n\), the derivative order \(r\), and
		the number of factors in a matrix product. Structural constants used throughout the induction are denoted by fixed
		symbols
		(\(\beta,c_0,C_0,B_0,C_{\tan}\)).
	\end{remark}
	
	\subsection{Return gap blocks and singular directions}
	\label{sec3}
	
	For an open interval \(J\subset\T\) and a phase \(x\in J\), let
	\begin{equation}\label{eq5}
		\begin{aligned}
			r_J^+(x)&=\min\{k\geq1:T^kx\in J\},\\
			r_J^-(x)&=\min\{k\geq1:T^{-k}x\in J\}
		\end{aligned}
	\end{equation}
	denote the forward and backward first-return times to \(J\). A
	\emph{joint return atom} of \(J\) is a maximal open subinterval of \(J\)
	on which both \(r_J^+\) and \(r_J^-\) are constant. On each such atom,
	the forward and backward \emph{open gap blocks}, excluding the transfer
	matrix at the central visit \(x\), are defined by
	\begin{equation*}
		\begin{split}
			B_J^+(E,x;v)
			&=
			A_{v,E}^{(r_J^+(x)-1)}(Tx),\\
			B_J^-(E,x;v)
			&=
			A_{v,E}^{(r_J^-(x)-1)}
			\bigl(T^{-r_J^-(x)+1}x\bigr).
		\end{split}
	\end{equation*}
	Whenever \(\|B_J^\pm(E,x;v)\|>1\) on a joint return atom, the singular
	values \(\Lambda_J^\pm=\|B_J^\pm\|>1>(\Lambda_J^\pm)^{-1}\) are simple,
	and the block admits the singular value decomposition (SVD)
	\begin{equation}\label{eq6}
		\begin{split}
			B_J^+&=\epsilon_J^+R_{\theta_J^+}D_{\Lambda_J^+}R_{\phi_J^+},\\
			B_J^-&=\epsilon_J^-R_{\theta_J^-}D_{\Lambda_J^-}R_{\phi_J^-},
		\end{split}
		\qquad
		D_\Lambda=\diag(\Lambda,\Lambda^{-1}),
	\end{equation}
	where \(R_\vartheta=\left(\begin{smallmatrix}\cos\vartheta&-\sin\vartheta\\
		\sin\vartheta&\cos\vartheta\end{smallmatrix}\right)\) denotes
	counterclockwise rotation by \(\vartheta\), and
	\(\epsilon_J^\pm\in\{\pm1\}\) are locally constant signs on each
	connected component of the domain. Because singular directions are
	projective lines in \(\mathbb{RP}^1\), we may choose continuous real
	lifts \(\theta_J^\pm,\phi_J^\pm\) of the angles---taking the representative
	closest to \(0\) whenever a singular direction is close to the
	coordinate axes---and absorb the resulting orientation choice into the
	locally constant signs \(\epsilon_J^\pm\in\{\pm1\}\). These signs do not
	affect operator norms, projective actions, or derivatives with respect to
	\((E,x)\).
	
	\begin{lemma}
		\label{lem1}
		Let \(v\in C^\infty(\T,\R)\), let \(\mathcal A\subset J\) be a joint
		return atom, let \(I\Subset\mathcal A\) be a compactly contained
		subinterval, and let \(\mathcal E\subset\R\) be a compact energy
		interval. Suppose that for one of the two gap blocks \(B=B_J^\pm\),
		\[
		\inf_{(E,x)\in\mathcal E\times I}
		\|B(E,x;v)\|\geq1+\gamma_0
		\]
		for some \(\gamma_0>0\). Then the top singular value \(\Lambda=\|B\|\)
		and any continuous lifts of the singular angles admit \(C^\infty\)
		branches on \(\mathcal E\times I\).
	\end{lemma}
	
	\begin{proof}
		Since \(r_J^\pm\) is constant on \(\mathcal A\), \(B(E,x;v)\) is a fixed
		finite product of smooth \(\mathrm{SL}(2,\R)\)-valued maps. The symmetric
		positive-definite matrix \(Q=B^{\mathsf T}B\) has eigenvalues
		\(\Lambda^2\) and \(\Lambda^{-2}\) separated by a uniform spectral gap
		\(\Lambda^2-\Lambda^{-2}\geq(1+\gamma_0)^2-(1+\gamma_0)^{-2}>0\). Thus
		\(\Lambda\) and the orthogonal eigenprojection
		\[
		\Pi_+=\frac{Q-\Lambda^{-2}\Id}{\Lambda^2-\Lambda^{-2}}
		\]
		depend smoothly on \((E,x)\in\mathcal E\times I\). Since the rectangle
		\(\mathcal E\times I\) is contractible, the rank-one bundle
		\(\operatorname{ran}\Pi_+\) admits a smooth unit section \(u(E,x)\).
		Completing \(u\) and \(Bu/\Lambda\) to positively oriented orthonormal
		frames yields smooth SVD angles on \(\mathcal E\times I\), with
		\(\epsilon_J^\pm\in\{\pm1\}\) locally constant according to the chosen
		lifts.
	\end{proof}
	
	Substituting the factorizations \eqref{eq6} into the two-gap
	return product centered at \(x\in J\) gives
	\begin{equation*}
		\begin{split}
			B_J^+(E,x;v)A_{v,E}(x)B_J^-(E,x;v)
			={}&\epsilon_J^+\epsilon_J^-R_{\theta_J^+}D_{\Lambda_J^+}
			C_J(E,x;v)
			D_{\Lambda_J^-}R_{\phi_J^-},\\
			C_J(E,x;v)
			={}&R_{\phi_J^+}A_{v,E}(x)R_{\theta_J^-}.
		\end{split}
	\end{equation*}
	The \((1,1)\)-entry of the connecting matrix \(C_J(E,x;v)\) measures the
	coupling between the expanding output direction of \(B_J^-\) under
	\(A_{v,E}(x)\) and the expanding input direction of \(B_J^+\).
	
	\begin{lemma}
		\label{lem2}
		Suppose that on a domain in \(\mathcal E\times\mathcal A\),
		\begin{equation}\label{lem2a}
			|\cos\phi_J^+(E,x)\cos\theta_J^-(E,x)|\geq c_0>0.
		\end{equation}
		Set \(a_J(E,x;v):=(C_J(E,x;v))_{11}\) and define the normalized
		connecting function
		\begin{equation}\label{lem2b}
			g_J(E,x;v)
			:=
			\frac{a_J(E,x;v)}
			{\cos\phi_J^+(E,x)\cos\theta_J^-(E,x)}.
		\end{equation}
		Then
		\begin{equation}\label{lem2c}
			g_J(E,x;v)
			=E-\lambda v(x)
			-\tan\phi_J^+(E,x)-\tan\theta_J^-(E,x).
		\end{equation}
		In particular, \(a_J(E,x;v)=0\) if and only if \(g_J(E,x;v)=0\).
	\end{lemma}
	
	\begin{proof}
		Direct matrix multiplication yields
		\[
		(R_\phi A_{v,E}(x)R_\theta)_{11}
		=(E-\lambda v(x))\cos\phi\cos\theta-\sin(\phi+\theta).
		\]
		Dividing both sides by \(\cos\phi\cos\theta\) and using the identity
		\(\sin(\phi+\theta)/(\cos\phi\cos\theta)=\tan\phi+\tan\theta\) gives
		\eqref{lem2c}.
	\end{proof}
	
	\begin{remark}
		Shifting either \(\phi_J^+\) or \(\theta_J^-\) by an integer multiple of
		\(\pi\) flips the signs of both \(a_J\) and
		\(\cos\phi_J^+\cos\theta_J^-\) simultaneously, leaving \(g_J(E,x;v)\)
		and \eqref{lem2c} independent of the choice of angle lifts.
	\end{remark}
	
	\subsection{Affine response under localized perturbations}
	
	Formula \eqref{lem2c} makes the affine response to a perturbation
	inside the return base \(J\) explicit.
	
	\begin{lemma}
		\label{lem3}
		Let \(\delta\in C^0(\T,\R)\) satisfy \(\supp\delta\Subset J\). Then for
		every \(x\in J\) and every \(E\in\R\),
		\begin{equation}\label{lem3a}
			B_J^\pm(E,x;v+\delta)=B_J^\pm(E,x;v).
		\end{equation}
		Consequently, on every joint return atom of \(J\) where
		\eqref{lem2a} holds,
		\begin{equation}\label{lem3b}
			g_J(E,x;v+\delta)
			=g_J(E,x;v)-\lambda\delta(x).
		\end{equation}
	\end{lemma}
	
	\begin{proof}
		By the definition \eqref{eq5} of the first-return times, the
		intermediate orbit points \(T^{\pm k}x\) for \(1\leq k<r_J^\pm(x)\) lie
		outside \(J\), where \(\delta\) vanishes identically. Thus every factor
		in \(B_J^\pm(E,x;v)\) is unchanged by adding \(\delta\), proving
		\eqref{lem3a}. Since the SVD angles \(\phi_J^+\) and
		\(\theta_J^-\) depend only on \(B_J^\pm\), they are likewise unchanged,
		so \eqref{lem2c} yields \eqref{lem3b}.
	\end{proof}
	
	\begin{remark}
		When comparing two nested return bases \(J_{n+1}\subset J_n\) at the same
		potential \(v\) and point \((E,x)\in\R\times J_{n+1}\), subtracting
		\eqref{lem2c} at the two scales eliminates the explicit potential
		term \(E-\lambda v(x)\), leaving only differences of the tangents of
		the singular angles:
		\[
		g_{J_{n+1}}(E,x;v)-g_{J_n}(E,x;v)
		=-\bigl(\tan\phi_{J_{n+1}}^+-\tan\phi_{J_n}^+\bigr)
		-\bigl(\tan\theta_{J_{n+1}}^--\tan\theta_{J_n}^-\bigr).
		\]
		Uniform lower bounds on \(\cos\phi_{J_k}^+\) and \(\cos\theta_{J_k}^-\)
		at both scales convert Gevrey estimates for the angle increments into
		estimates for these tangent differences.
	\end{remark}
	
	\subsection{Dual-number Gevrey Banach algebras and two-block concatenation}
	\label{sec4}
	
	We will use the following elementary factorial bounds: for any \(a>0\),
	\(d\geq0\), and \(y\geq0\),
	\begin{align}
		\sup_{r\geq0}\frac{(1+r)^2y^r}{(r!)^a}
		&\leq C(a)\exp\bigl\{C(a)y^{1/a}\bigr\},
		\label{eq7a}\\
		\sum_{r\geq0}\frac{(1+r)^dy^r}{(r!)^a}
		&\leq C(a,d)\exp\bigl\{C(a,d)y^{1/a}\bigr\}.
		\label{eq7b}
	\end{align}
	Indeed, for \(r\geq M:=\lceil(2^{d+1}y)^{1/a}\rceil\), the ratio of
	consecutive terms in \eqref{eq7b} is at most \(1/2\),
	so the tail is dominated by the \(M\)-th term. The first \(M+1\)
	terms are controlled by \eqref{eq7a}.
	
	To control functions of \((E,x)\) together with their first energy
	derivative on a parameter rectangle \(\mathcal R=\mathcal E\times I\), we
	define the mixed norm
	\begin{equation}\label{eq8}
		|F|_{\mathfrak s,K;1;\mathcal R}
		:=
		\sup_{E\in\mathcal E}
		\bigl(
		|F(E,\cdot)|_{\mathfrak s,K;I}
		+|\partial_EF(E,\cdot)|_{\mathfrak s,K;I}
		\bigr).
	\end{equation}
	When \(I\) is a finite disjoint union of open subintervals, the same notation denotes the maximum of the component
	norms. When the domain \(\mathcal R\) is clear from the context, we
	abbreviate \(|F|_{\mathfrak s,K;1;\mathcal R}\) as
	\(|F|_{\mathfrak s,K;1}\).
	
	Algebraically, \eqref{eq8} is the uniform norm on
	\(\mathcal E\) with values in the first-order jet (dual-number) Banach
	algebra
	\begin{equation*}
		\mathcal G_{\mathfrak s,K}^{(1)}(I)
		:=G^{\mathfrak s,K}(I)
		\oplus\epsilon G^{\mathfrak s,K}(I),
		\qquad \epsilon^2=0,
	\end{equation*}
	equipped with the norm
	\(|f+\epsilon g|_{\mathfrak s,K;1}:=|f|_{\mathfrak s,K;I}+|g|_{\mathfrak s,K;I}\)
	under the jet embedding
	\(F(E,\cdot)\mapsto F(E,\cdot)+\epsilon\partial_EF(E,\cdot)\). The next
	two lemmas show that analytic composition and hyperbolic two-block
	concatenation preserve \(\mathcal G_{\mathfrak s,K}^{(1)}(I)\) with the
	same Gevrey parameter \(K\).
	
	\begin{lemma}
		\label{lem4}
		Let \(\mathfrak s>1\), \(K\geq1\), and let \(I\subset\R\) be an interval.
		\begin{enumerate}
			\item For scalar- or matrix-valued functions, both \(G^{\mathfrak s,K}(I)\)
			and \(\mathcal G_{\mathfrak s,K}^{(1)}(I)\) are Banach algebras.
			\item Let \(F\in G^{\mathfrak s,K}(I)\) take values in a compact subset
			\(\mathcal K\) of the domain of a real-analytic map \(\Psi\). Then
			\(\Psi\circ F\in G^{\mathfrak s,K}(I)\) with the \emph{same}
			parameter \(K\), and \(|\Psi\circ F|_{\mathfrak s,K;I}\) is
			bounded by a constant depending only on \(\mathfrak s\), a complex
			neighborhood of \(\mathcal K\) on which \(\Psi\) extends
			holomorphically, the supremum of that extension, and
			\(|F|_{\mathfrak s,K;I}\). The same property
			holds in \(\mathcal G_{\mathfrak s,K}^{(1)}(I)\).
			\item In particular, inversion, square roots, logarithms, local SVD maps
			on the locus of simple singular values, and the trigonometric
			functions \(\tan\) and \(\sec^2\) on any compact interval avoiding
			odd multiples of \(\pi/2\) preserve \(\mathcal G_{\mathfrak s,K}^{(1)}(I)\)
			with the same parameter \(K\).
		\end{enumerate}
	\end{lemma}
	
	\begin{proof}
		By the Leibniz rule and the elementary inequality
		\[
		\sup_{r\geq0}\sum_{j=0}^r
		\frac{(1+r)^2}{(1+j)^2(1+r-j)^2}
		\binom rj^{1-\mathfrak s}
		\leq4\sum_{k\geq1}k^{-2}=\frac{2\pi^2}{3}<\frac{4\pi^2}{3},
		\]
		the norm \eqref{def1a} satisfies
		\eqref{eq4}. For the dual-number algebra
		\(\mathcal G_{\mathfrak s,K}^{(1)}(I)\), submultiplicativity follows from
		\(|(f_1+\epsilon g_1)(f_2+\epsilon g_2)|_{\mathfrak s,K;1}
		\leq |f_1f_2|+|f_1g_2|+|g_1f_2|
		\leq(|f_1|+|g_1|)(|f_2|+|g_2|)\).
		
		For part~(2), choose \(d>0\) and \(M<\infty\) such that \(\Psi\) extends
		holomorphically to the complex \(d\)-neighborhood of \(\mathcal K\) with
		\(\|\Psi\|\leq M\). Cauchy's estimates give
		\(\|D^j\Psi(\mathcal K)\|\leq Md^{-j}j!\) for all \(j\geq0\). In
		Fa\`a di Bruno's formula for \(\partial_x^r(\Psi\circ F)\), a partition
		of \(r\) has multiplicities \(m_1,\ldots,m_r\) satisfying
		\(\sum_{k=1}^r km_k=r\) and \(\sum_{k=1}^r m_k=j\). The required estimate
		then reduces to the combinatorial convolution
		\[
		(1+r)^2
		\sum_{\substack{m_1+2m_2+\cdots+rm_r=r\\
				m_1+\cdots+m_r=j}}
		\frac{j!}{\prod_km_k!}
		\left(\frac{\prod_k(k!)^{m_k}}{r!}\right)^{\mathfrak s-1}
		\prod_k(1+k)^{-2m_k}
		\leq\frac{C(\mathfrak s)^j}{(j!)^{\mathfrak s-1}}.
		\]
		Indeed, rewriting the sum over ordered compositions
		\(k_1+\cdots+k_j=r\) (\(k_i\geq1\)), log-convexity of the factorial
		implies \(k_1!\cdots k_j!\leq(r-j+1)!\leq r!/j!\), so the factorial
		ratio is at most \((j!)^{1-\mathfrak s}\), while iterating
		\(\sum_{k=1}^{r-1}(1+k)^{-2}(1+r-k)^{-2}\leq C(1+r)^{-2}\) \(j-1\) times
		contributes \(C^{j-1}(1+r)^{-2}\). Consequently,
		\[
		|\Psi\circ F|_{\mathfrak s,K;I}
		\leq C(\mathfrak s)M
		\sum_{j\geq0}
		\frac{\bigl(C(\mathfrak s)|F|_{\mathfrak s,K;I}/d\bigr)^j}
		{(j!)^{\mathfrak s-1}}<\infty,
		\]
		with the same factor \(K^r\) at every derivative order. In the
		dual-number algebra,
		\(\Psi(F+\epsilon\partial_EF)=\Psi(F)+\epsilon D\Psi(F)\partial_EF\), so
		applying the spatial estimate to both \(\Psi\) and \(D\Psi\) and using
		submultiplicativity yields the mixed-norm assertion. Part~(3) follows
		because the SVD maps are real analytic wherever the singular
		values are distinct and the cosine lower bounds keep \(\tan\) and
		\(\sec^2\) away from their poles.
	\end{proof}
	
	\begin{lemma}
		\label{lem5}
		On an interval \(I\), consider a matrix product of the form
		\[
		X=D_\nu U D_\mu,
		\qquad
		U=\begin{pmatrix}a&b\\c&d\end{pmatrix},
		\]
		where \(\mu,\nu>1\), \(a\ne0\), and all matrix entries belong to
		\(\mathcal G_{\mathfrak s,K}^{(1)}(I)\). Set
		\begin{equation}\label{lem5a}
			\mathcal E
			:=\max\bigl\{|\mu^{-1}|_{\mathfrak s,K;1},
			|\nu^{-1}|_{\mathfrak s,K;1}\bigr\},
			\qquad
			\mathcal H
			:=1+|a^{-1}|_{\mathfrak s,K;1}
			+|b/a|_{\mathfrak s,K;1}
			+|c/a|_{\mathfrak s,K;1}
			+|d/a|_{\mathfrak s,K;1}.
		\end{equation}
		If the hyperbolic domination condition
		\begin{equation}\label{lem5b}
			\mathcal H\mathcal E^2\leq 2^{-6}
		\end{equation}
		holds, then the top singular value \(\Lambda(X)=\|X\|\) is simple, and
		the tangent coordinates \(z_L,z_R\) of the left and right expanding
		singular directions of \(X\) relative to the exterior diagonal axes
		satisfy, in the \emph{same} algebra \(\mathcal G_{\mathfrak s,K}^{(1)}(I)\),
		\begin{align}
			|\Lambda(X)^{-1}|_{\mathfrak s,K;1}
			&\leq C|a^{-1}|_{\mathfrak s,K;1}
			|\mu^{-1}|_{\mathfrak s,K;1}
			|\nu^{-1}|_{\mathfrak s,K;1},
			\label{lem5c}\\
			|z_L|_{\mathfrak s,K;1}
			&\leq C\mathcal H
			|\nu^{-1}|_{\mathfrak s,K;1}^{2},
			\label{lem5d}\\
			|z_R|_{\mathfrak s,K;1}
			&\leq C\mathcal H
			|\mu^{-1}|_{\mathfrak s,K;1}^{2},
			\label{lem5e}
		\end{align}
		for an absolute constant \(C>0\). The corresponding angle increments
		\(\arctan z_L\) and \(\arctan z_R\) obey the same bounds.
	\end{lemma}
	
	\begin{proof}
		Set \(\beta=b/a\), \(\gamma=c/a\), and \(\delta=d/a\). Factoring out the
		leading scalar \(a\mu\nu\) gives
		\begin{equation*}
			\frac{X}{a\mu\nu}
			=W:=\begin{pmatrix}1&x\\y&z\end{pmatrix},
			\qquad
			x=\beta\mu^{-2},\quad
			y=\gamma\nu^{-2},\quad
			z=\delta\mu^{-2}\nu^{-2}.
		\end{equation*}
		Since \(\mathcal H\geq1\), the condition \eqref{lem5b}
		implies \(\mathcal E^2\leq 2^{-6}\), and submultiplicativity in
		\(\mathcal G_{\mathfrak s,K}^{(1)}(I)\) yields
		\begin{equation}\label{eq9}
			|x|_{\mathfrak s,K;1}+|y|_{\mathfrak s,K;1}
			+|z|_{\mathfrak s,K;1}
			\leq 3\mathcal H\mathcal E^2
			\leq\frac{3}{64}<\frac1{16}.
		\end{equation}
		
		The entries of \(WW^{\mathsf T}\) are \(A_L=1+x^2\), \(B_L=y+xz\), and
		\(C_L=y^2+z^2\). The spectral gap of \(WW^{\mathsf T}\) and the slope of
		its top eigenvector are given explicitly by
		\[
		\Delta_L=\sqrt{(A_L-C_L)^2+4B_L^2},\qquad
		z_L=\frac{2B_L}{A_L-C_L+\Delta_L}.
		\]
		By \eqref{eq9}, the denominator
		\(A_L-C_L+\Delta_L\) differs from \(2\) by an element of norm at most
		\(1/4\) in \(\mathcal G_{\mathfrak s,K}^{(1)}(I)\). Expanding the square
		root and reciprocal in convergent binomial and Neumann series within
		\(\mathcal G_{\mathfrak s,K}^{(1)}(I)\) yields
		\[
		|z_L|_{\mathfrak s,K;1}
		\leq C|B_L|_{\mathfrak s,K;1}
		\leq C\bigl(|y|_{\mathfrak s,K;1}+|x|_{\mathfrak s,K;1}|z|_{\mathfrak s,K;1}\bigr)
		\leq C\mathcal H|\nu^{-1}|_{\mathfrak s,K;1}^2,
		\]
		proving \eqref{lem5d}. Applying the same calculation
		to \(W^{\mathsf T}W\) (which interchanges \(x\) and \(y\)) gives
		\eqref{lem5e}. If we write the SVD of an
		\(\mathrm{SL}(2,\R)\) block as \(X=\epsilon_X R_\Theta D_{\Lambda(X)}R_\Phi\)
		with the small angle lifts, then \(\epsilon_X=\operatorname{sgn}a\) and
		the right rotation angle increment is \(-\arctan z_R\); the locally
		constant sign \(\epsilon_X\) does not enter any of the norm or angle
		estimates.
		
		Similarly, the top singular value \(N(W)=\|W\|\) satisfies
		\(N(W)^2=(A_L+C_L+\Delta_L)/2\), so \(|N(W)^2-1|_{\mathfrak s,K;1}<1/4\),
		whence \(|N(W)^{-1}|_{\mathfrak s,K;1}\leq C\). Since \(a\) has constant
		sign on \(I\), \(\Lambda(X)^{-1}=|a|^{-1}\mu^{-1}\nu^{-1}N(W)^{-1}\),
		and submultiplicativity gives \eqref{lem5c}.
		Finally, \(\arctan z=z\int_0^1(1+t^2z^2)^{-1}\,dt\) satisfies the same
		bounds by \eqref{eq9}.
	\end{proof}
	
	For a Schr\"odinger connector, the rotated one-step transfer matrix has
	the form
	\[
	U=R_\Phi A_{v,E}(x)R_\Theta
	=\begin{pmatrix}a&b\\c&d\end{pmatrix},
	\qquad
	Q=E-\lambda v(x),\quad
	p=\tan\Phi,\quad t=\tan\Theta,
	\quad g=Q-p-t,
	\]
	on a domain where \(\cos\Phi\cos\Theta\neq0\) and \(g\neq0\). Direct
	matrix multiplication and substitution of \(Q=g+p+t\) yield the exact
	ratios
	\begin{equation}\label{eq10}
		\begin{aligned}
			a^{-1}&=(\cos\Phi\cos\Theta)^{-1}g^{-1},\\
			\frac ba&=-t-(1+t^2)g^{-1},\\
			\frac ca&=p+(1+p^2)g^{-1},\\
			\frac da&=-pt-(p+t)(1+pt)g^{-1}.
		\end{aligned}
	\end{equation}

	\section{Continued-fraction return geometry and flat profiles}
	\label{sec5}
	
	Throughout Sections~\ref{sec5}--\ref{sec8},
	we fix an irrational frequency \(\alpha\) and an admissible Diophantine
	pair \((\gamma,\tau)\) satisfying
	\begin{gather}
		\|k\alpha\|_{\T}\geq\gamma |k|^{-\tau},
		\qquad k\in\Z\setminus\{0\},\qquad \tau\geq1,
		\label{eqDC}\\
		q_{n+1}\leq C(\gamma,\tau) q_n^\tau.
		\label{eq11b}
	\end{gather}
	
	\subsection{Regularity exponents and scaling parameters}
	
	Fix a target Gevrey exponent
	\begin{equation*}
		\mathfrak s>1+\tau .
	\end{equation*}
	The auxiliary parameters are the flatness
	exponent \(\nu\) in \(h(x)=\exp(-|\sin\pi x|^{-\nu})\),
	an intermediate Gevrey index \(\sigma\in(1,\mathfrak s)\) for the cutoff
	functions, and the cost exponents \(\mathfrak b\) and \(\varkappa\)
	governing the Gevrey norms of the rescaled cutoffs and profile
	reciprocals. Explicitly, we set
	\begin{equation*}
		\begin{array}
			[c]{ll}%
			\delta=\dfrac{\mathfrak s-1-\tau}{4}>0,\qquad
			\sigma=1+\delta,\qquad
			\nu=\dfrac1{\tau+2\delta}, \\
			\mathfrak b=\dfrac1{\mathfrak s-\sigma}
			=\dfrac1{\tau+3\delta},\qquad
			\varkappa=\dfrac{1+\nu}{\mathfrak s}, %
		\end{array}
	\end{equation*}
	and fix an exponent \(\vartheta\) satisfying
	\begin{equation*}
		\tau\nu<\vartheta<1.
	\end{equation*}
	The identity \(\mathfrak s-1=\tau+4\delta\) gives
	\(\nu(\mathfrak s-1)>1\), equivalently \(\varkappa<\nu\).
	Consequently, these parameters satisfy the strict ordering
	\begin{equation}\label{eq12}
		\begin{gathered}
			\frac1{\mathfrak s-1}<\nu<\frac1\tau,\qquad
			1<\sigma<\mathfrak s,\qquad
			0<\mathfrak b<\nu,\qquad 0<\varkappa<\nu,\\
			\tau\max\{\mathfrak b,\varkappa\}
			<\tau\nu<\vartheta<1.
		\end{gathered}
	\end{equation}
	The interval \(\bigl(1/(\mathfrak s-1),1/\tau\bigr)\) containing
	\(\nu\) is nonempty if and only if \(\mathfrak s>1+\tau\). This condition
	accounts for the regularity threshold in Theorem~\ref{thm1}.
	
	\subsection{Continued-fraction return bases and good scales}
	
	Let \(\alpha=[\mathfrak a_0;\mathfrak a_1,\mathfrak a_2,\ldots]\) be the
	continued-fraction expansion of \(\alpha\), with principal convergents
	\(p_s/q_s\). For each \(s\geq1\), define the canonical symmetric interval
	\begin{equation*}
		\eta_s=\|q_s\alpha\|_{\T},\qquad
		\widehat\rho_s=\frac{\eta_s+\eta_{s+1}}2,
		\qquad \widehat I_s=(-\widehat\rho_s,\widehat\rho_s),
		\qquad d_s=\frac{\eta_s-\eta_{s+1}}2.
	\end{equation*}
	To keep the internal partition points \(\pm d_s\) of \(\widehat I_s\)
	uniformly separated from the origin relative to \(\widehat\rho_s\) (which requires avoiding \(\eta_s\approx\eta_{s+1}\)) while ensuring that at least one of any two consecutive indices is selected, we call a continued-fraction index \(s\)
	\emph{good} if
	\begin{equation*}
		\mathfrak a_{s+2}\geq2,
		\quad\text{or}\quad
		\bigl(\mathfrak a_{s+2}=1\ \text{ and }\ \mathfrak a_{s+3}\leq2\bigr).
	\end{equation*}
	We enumerate all sufficiently large good indices in increasing order as
	\(g_1<g_2<\cdots\), and define the scale-\(n\) return data by
	\begin{equation*}
		\begin{aligned}
			\rho_n&:=\widehat\rho_{g_n},
			& I_n&:=\widehat I_{g_n}=(-\rho_n,\rho_n),\\
			Q_n^-&:=q_{g_n},
			& Q_n^+&:=q_{g_n+1},\\
			R_{n,S}&:=Q_n^--1,
			& R_{n,L}&:=Q_n^+-1.
		\end{aligned}
	\end{equation*}
	After discarding finitely many initial indices if necessary, we have
	\(R_{n,S}\geq Q_n^-/2\) and \(R_{n,L}\geq Q_n^+/2\). Moreover, since
	\(g_{n+1}\geq g_n+1\), the return lengths satisfy the ordering
	\(R_{n,S}<R_{n,L}\leq R_{n+1,S}<R_{n+1,L}\).
	
	\begin{definition}
		\label{def2}
		For each scale \(n\geq1\), let \(\mathcal P_n:=\{-d_{g_n},d_{g_n}\}\subset I_n\)
		denote the two internal return-partition points of \(I_n\). We remove
		the full orbits of all partition and boundary points by setting
		\begin{equation*}
			\mathcal Z_{\rm ret}
			:=\bigcup_{n\geq1}\ \bigcup_{k\in\Z}
			T^{-k}\bigl(\mathcal P_n\cup\partial I_n\bigr),
			\qquad
			\Omega_{\rm ret}:=\T\setminus\mathcal Z_{\rm ret}.
		\end{equation*}
		A \emph{regular \(I_n\)-return atom} is one of the three connected
		components of \(I_n\setminus\mathcal P_n\). For \(x\in\Omega_{\rm ret}\),
		an integer \(t\in\Z\) with \(T^tx\in I_n\) is called a \emph{regular
			\(I_n\)-return time} (or node); if \(t<t'\) are consecutive return times
		to \(I_n\), the transfer matrix product along \(T^{t+1}x,\ldots,T^{t'-1}x\)
		is a \emph{regular open \(I_n\)-gap block} of physical length
		\(\ell=t'-t-1\).
	\end{definition}
	
	Since \(\mathcal Z_{\rm ret}\) is countable and \(T\)-invariant, its
	complement \(\Omega_{\rm ret}\) is a dense, \(T\)-invariant set of full
	Lebesgue measure, so all phase integrals defining \(L_v(E)\) may be
	restricted to \(\Omega_{\rm ret}\).
	
	\begin{lemma}
		\label{lem6}
		Consecutive good indices satisfy \(1\leq g_{n+1}-g_n\leq2\). At every
		good index \(s\),
		\begin{equation}\label{lem6a}
			d_s\geq \frac17\widehat\rho_s,
			\qquad \widehat\rho_{s+1}\leq\frac23\widehat\rho_s.
		\end{equation}
		Consequently, for all \(n\geq1\),
		\begin{equation}\label{lem6b}
			\rho_{n+1}\leq\frac23\rho_n,
			\qquad
			\rho_n^{-1}\leq C(\gamma,\tau)(Q_n^-)^\tau,
			\qquad
			\rho_{n+1}^{-1}\leq C(\gamma,\tau)(Q_n^+)^\tau.
		\end{equation}
	\end{lemma}
	
	\begin{proof}
		If \(s\) is not good, then \(\mathfrak a_{s+2}=1\) and
		\(\mathfrak a_{s+3}\geq3\), so \(s+1\) is good. Hence
		\(g_{n+1}-g_n\in\{1,2\}\). Setting
		\(t=\eta_{s+2}/\eta_{s+1}\in(0,1)\) and using the recurrence
		\(\eta_s=\mathfrak a_{s+2}\eta_{s+1}+\eta_{s+2}\), we have
		\[
		\frac{d_s}{\widehat\rho_s}
		=\frac{\mathfrak a_{s+2}-1+t}{\mathfrak a_{s+2}+1+t},
		\qquad
		\frac{\widehat\rho_{s+1}}{\widehat\rho_s}
		=\frac{1+t}{\mathfrak a_{s+2}+1+t}.
		\]
		When \(\mathfrak a_{s+2}\geq2\), these two fractions are at least \(1/3\)
		and at most \(1/2\), respectively. When \(\mathfrak a_{s+2}=1\) and
		\(\mathfrak a_{s+3}\leq2\), we have
		\(\eta_{s+1}=\mathfrak a_{s+3}\eta_{s+2}+\eta_{s+3}<3\eta_{s+2}\), so
		\(t>1/3\), and the two fractions are bounded below by \(1/7\) and above
		by \(2/3\). This proves \eqref{lem6a}.
		
		The standard continued-fraction estimates
		\((2q_{s+1})^{-1}<\eta_s<q_{s+1}^{-1}\) imply
		\(q_{s+1}<\widehat\rho_s^{-1}<4q_{s+1}\). Together with the Diophantine
		bound \eqref{eq11b}, this gives
		\(\rho_n^{-1}<4q_{g_n+1}\leq C(\gamma,\tau)(Q_n^-)^\tau\). If
		\(g_{n+1}=g_n+1\), the same bound at \(s=g_n+1\) gives
		\(\rho_{n+1}^{-1}\leq C(\gamma,\tau)(Q_n^+)^\tau\). If
		\(g_{n+1}=g_n+2\), then \(g_n+1\) is not good, so
		\(\mathfrak a_{g_n+3}=1\) and hence
		\[
		q_{g_n+3}=q_{g_n+2}+q_{g_n+1}
		\leq C(\gamma,\tau)(Q_n^+)^\tau+Q_n^+
		\leq C'(\gamma,\tau)(Q_n^+)^\tau,
		\]
		whence \(\rho_{n+1}^{-1}=\widehat\rho_{g_n+2}^{-1}<4q_{g_n+3}\leq C(\gamma,\tau)(Q_n^+)^\tau\).
	\end{proof}
	
	Combining \eqref{eq12} and
	\eqref{lem6b} yields the regional sublinear
	bounds
	\begin{equation}\label{eq13}
		\begin{aligned}
			\rho_n^{-\nu}+\rho_n^{-\mathfrak b}
			+\rho_n^{-\varkappa}
			&\leq C(\gamma,\tau) R_{n,S}^{\vartheta}=:C(\gamma,\tau) Y_{n,S},\\
			\rho_{n+1}^{-\nu}+\rho_{n+1}^{-\mathfrak b}
			+\rho_{n+1}^{-\varkappa}
			&\leq C(\gamma,\tau) R_{n,L}^{\vartheta}=:C(\gamma,\tau) Y_{n,L}.
		\end{aligned}
	\end{equation}
	Since the convergent denominators \(q_s\) grow at least exponentially,
	the sequence of errors
	\begin{equation*}
		\Delta_n^*:=
		R_{n,S}^{\vartheta-1}+R_{n,L}^{\vartheta-1}
		+R_{n,S}^{-1}+R_{n,L}^{-1}
	\end{equation*}
	is summable, with \(\sum_{n\geq n_0}\Delta_n^*\to0\) as
	\(n_0\to\infty\).
	
	\subsection{Decomposition of return blocks and entrance times}
	
	Each open return gap at scale \(n+1\) decomposes into open return gaps
	at scale \(n\) separated by single one-step transfer matrices at the
	intermediate visits to \(I_n\setminus I_{n+1}\).
	
	\begin{lemma}
		\label{lem7}
		At scale \(n\), the base \(I_n=\widehat I_{g_n}\) is partitioned by
		\(\mathcal P_n=\{-d_{g_n},d_{g_n}\}\) into three regular return atoms:
		\begin{enumerate}
			\item On the central atom \(\mathcal A_n^0:=(-d_{g_n},d_{g_n})\), both
			forward and backward first-return times to \(I_n\) equal
			\(Q_n^+\) (so both adjacent open gaps have long length
			\(R_{n,L}=Q_n^+-1\)).
			\item On each of the two side atoms comprising
			\(I_n\setminus\overline{\mathcal A_n^0}\), the unordered pair of
			forward and backward return times is \(\{Q_n^-,Q_n^+\}\) (one
			short gap of length \(R_{n,S}=Q_n^--1\) and one long gap of length
			\(R_{n,L}=Q_n^+-1\)).
		\end{enumerate}
		
		Let \(x\in\Omega_{\rm ret}\), let \(t<t'\) be consecutive return
		times of the orbit of \(x\) to \(I_{n+1}\), and list all intermediate
		return times to \(I_n\) in \([t,t']\) as
		\(t=t_0<t_1<\cdots<t_m<t_{m+1}=t'\). Setting \(y_i=T^{t_i}x\in I_n\setminus I_{n+1}\)
		(\(1\leq i\leq m\)) and
		\(G_i=A_{v,E}^{(t_i-t_{i-1}-1)}(T^{t_{i-1}+1}x)\), we have the exact
		factorization
		\begin{equation}\label{lem7a}
			A_{v,E}^{(t'-t-1)}(T^{t+1}x)
			=G_{m+1}A_{v,E}(y_m)G_m\cdots A_{v,E}(y_1)G_1,
		\end{equation}
		where each \(G_i\) is a regular open \(I_n\)-gap block of type
		\(T_i\in\{S,L\}\) and the lengths satisfy
		\begin{equation}\label{lem7b}
			\sum_{i=1}^{m+1}(t_i-t_{i-1}-1)+m=t'-t-1.
		\end{equation}
		If we assign each side-atom visit \(y_i\in I_n\setminus\overline{\mathcal A_n^0}\)
		(where \(|y_i|\geq\frac17\rho_n\)) to its unique adjacent short gap, and
		each central-atom visit \(y_i\in\mathcal A_n^0\setminus I_{n+1}\) (where
		\(|y_i|\geq\rho_{n+1}\)) to its preceding long gap \(G_i\), then every
		connector is assigned to an adjacent gap of minimal length
		\(\min\{\ell_i,\ell_{i+1}\}\) and each gap \(G_j\) receives at most two
		assignments.
	\end{lemma}
	
	\begin{proof}
		Set \(s=g_n\), \(a=\eta_s\), and \(b=\eta_{s+1}\), so that
		\(\widehat\rho_s=(a+b)/2\) and \(d_s=(a-b)/2>0\). Reversing orientation
		on \(\T\) if necessary, we may assume \(q_s\alpha-p_s=a\) and
		\(q_{s+1}\alpha-p_{s+1}=-b\). Any return to \(\widehat I_s\) at time
		\(0<k<q_{s+1}\) corresponds to an integer pair \((k,l)\) with
		\(|k\alpha-l|<a+b\). Since \((q_s,p_s)\) and \((q_{s+1},p_{s+1})\) form
		a \(\Z\)-basis of \(\Z^2\), we may write
		\((k,l)=u(q_s,p_s)+v(q_{s+1},p_{s+1})\), so that \(k\alpha-l=ua-vb\).
		Because \(0<k<q_{s+1}\), if \(v\geq1\) then \(u\leq-1\) and
		\(ua-vb\leq-a-b\), whereas if \(v\leq-1\) then \(u\geq1\) and
		\(ua-vb\geq a+b\). Thus \(v=0\), and since \(2a>a+b\) we must have
		\(u=1\), leaving \(k=q_s\) as the only possible return time strictly
		smaller than \(q_{s+1}\).
		
		For \(x\in(-d_s,d_s)\), we have \(|x\pm b|<d_s+b=\widehat\rho_s\) and
		\(|x\pm a|>a-d_s=\widehat\rho_s\), so both forward and backward first
		return times equal \(q_{s+1}\). For \(x\in(d_s,\widehat\rho_s)\), we
		have \(x-a\in(-\widehat\rho_s,-d_s)\), \(x+a\notin\widehat I_s\), and
		\(x-b\in(-\widehat\rho_s,d_s)\), so the backward and forward return
		times are \(q_s\) and \(q_{s+1}\), respectively; by symmetry, they are
		\(q_{s+1}\) and \(q_s\) on \((-\widehat\rho_s,-d_s)\).
		
		The factorization \eqref{lem7a} and length
		identity \eqref{lem7b} follow by
		splitting the orbit segment \((t,t')\) at all visits to \(I_n\). By
		\eqref{lem6a}, any visit in a side atom satisfies
		\(|y_i|\geq d_{g_n}\geq\frac17\rho_n\), while any visit in
		\(\mathcal A_n^0\setminus I_{n+1}\) satisfies \(|y_i|\geq\rho_{n+1}\).
		Finally, each gap \(G_j\) can receive an assignment at most from its two
		endpoints \(y_{j-1}\) and \(y_j\).
	\end{proof}
	
	The same return geometry gives bounds on the entrance time and
	Lebesgue measure of \(I_n\).
	
	\begin{lemma}
		\label{lem8}
		For every \(n\geq1\), the maximum forward (or backward) entrance time to
		\(I_n\) satisfies
		\begin{align}
			H_n&:=\sup_{x\in\T}\min\{k\geq0:T^kx\in I_n\}
			\leq Q_n^-+Q_n^+<2\rho_n^{-1},
			\label{lem8a}\\
			1&<Q_n^+|I_n|<2.
			\label{lem8b}
		\end{align}
	\end{lemma}
	
	\begin{proof}
		Write \(s=g_n\). The \(q_s+q_{s+1}\) points \(\{k\alpha\}_{k=0}^{q_s+q_{s+1}-1}\)
		partition \(\T\) into intervals of lengths \(\eta_s\) and \(\eta_{s+1}\),
		with \(q_{s+1}\) gaps of length \(\eta_s\) and \(q_s\) gaps of length
		\(\eta_{s+1}\) summing to \(q_{s+1}\eta_s+q_s\eta_{s+1}=1\). Since
		\(|I_n|=2\rho_n=\eta_s+\eta_{s+1}\) is strictly larger than every gap
		of this partition, every translate of \(I_n\) contains at least one point
		of \(\{k\alpha\}_{k=0}^{q_s+q_{s+1}-1}\), which gives
		\(H_n\leq Q_n^-+Q_n^+<2\rho_n^{-1}\). Moreover,
		\[
		Q_n^+|I_n|=q_{s+1}(\eta_s+\eta_{s+1})
		=1+(q_{s+1}-q_s)\eta_{s+1}\in(1,2),
		\]
		proving \eqref{lem8b}.
	\end{proof}
	
	\subsection{Flat profile, cutoff windows, and resonant energies}
	
	We build the potential around the smooth flat profile
	\begin{equation*}
		h(x)=
		\begin{cases}
			\exp\bigl\{-|\sin\pi x|^{-\nu}\bigr\},&x\notin\Z,\\
			0,&x\in\Z.
		\end{cases}
	\end{equation*}
	Since \(|\sin\pi x|\in[0,1]\), we have
	\(0\leq h(x)\leq e^{-1}<1/2\) for all \(x\in\T\).
	
	\begin{lemma}
		\label{lem9}
		There exists \(K_0=K_0(\nu)\geq1\) such that
		\begin{equation}\label{lem9a}
			h\in G^{1+1/\nu,K_0}(\T)
			\subset G^{\mathfrak s,K_0}(\T).
		\end{equation}
		Moreover, for any fixed \(0<c_0<C_0<1/2\), there exist constants
		\(c,C,\rho_0>0\) such that for all \(0<\rho\leq\rho_0\),
		\begin{equation}
			|h|_{\mathfrak s,K_0;
				\{c_0\rho\leq |x|\leq C_0\rho\}}
			\leq C e^{-c\rho^{-\nu}}.
			\label{lem9b}
		\end{equation}
		For the reciprocal estimate, set \(S(t):=|\sin\pi t|^{-\nu}\) and
		\(\mathcal A(r):=\{x\in\T:r\leq|x|\leq\rho_0\}\). Then
		\begin{equation}\label{lem9c}
			|h^{-1}|_{\mathfrak s,K;\mathcal A(r)}
			\leq C\exp\bigl\{S(r)+Cr^{-\varkappa}\bigr\},
			\qquad
			0<r\leq\rho_0,\quad K\geq K_0.
		\end{equation}
	\end{lemma}
	
	\begin{proof}
		On the positive and negative punctured half-neighborhoods of \(x=0\),
		\(h\) extends to the holomorphic branches
		\(H_\pm(z)=\exp\{-(\pm\sin\pi z)^{-\nu}\}\). For sufficiently small
		\(a>0\) and \(0<|x|<\rho_0\), the complex disc \(|z-x|\leq a|x|\) remains
		within a sector where \(c|x|^{-\nu}\leq\Re(\pm\sin\pi z)^{-\nu}\leq C|x|^{-\nu}\).
		Cauchy's integral formula on \(|z-x|=a|x|\) yields
		\begin{equation}\label{eq14}
			|\partial_x^rh(x)|
			\leq C^{r+1}r!|x|^{-r}e^{-c|x|^{-\nu}},
			\qquad r\geq0.
		\end{equation}
		Maximizing \(|x|^{-r}e^{-c|x|^{-\nu}}\) over \(|x|>0\) gives
		\(\|\partial_x^rh\|_{C^0(\T)}\leq C^{r+1}(r!)^{1+1/\nu}\), which implies
		\eqref{lem9a} since \(1+1/\nu<\mathfrak s\). On the
		annulus \(c_0\rho\leq|x|\leq C_0\rho\), \eqref{eq14}
		and the factorial bound \eqref{eq7a} give
		\[
		|h|_{\mathfrak s,K_0;\{c_0\rho\leq|x|\leq C_0\rho\}}
		\leq Ce^{-c\rho^{-\nu}}\sup_{r\geq0}\frac{(1+r)^2(C/\rho)^r}{(r!)^{\mathfrak s-1}}
		\leq C\exp\bigl\{-c\rho^{-\nu}+C\rho^{-1/(\mathfrak s-1)}\bigr\},
		\]
		which yields \eqref{lem9b} because
		\(1/(\mathfrak s-1)<\nu\).
		
		To prove \eqref{lem9c}, observe that on
		\(\mathcal A(r)\) the analytic function \(S(x)=|\sin\pi x|^{-\nu}\)
		satisfies \(|\partial_x^jS(x)|\leq C^{j+1}j!|x|^{-\nu-j}\) for all
		\(j\geq1\). Applying Fa\`a di Bruno's formula to \(e^{S(x)}\) and using
		the combinatorial identity
		\[
		\sum_{\substack{m_1+\cdots+m_q=k\\
				m_1+2m_2+\cdots+qm_q=q}}
		\frac1{m_1!\cdots m_q!}
		=\frac1{k!}\binom{q-1}{k-1}
		\]
		gives, for \(q\geq1\),
		\[
		|\partial_x^q e^{S(x)}|
		\leq e^{S(x)}q!(C/|x|)^q
		\sum_{k=1}^q
		\binom{q-1}{k-1}\frac{(C|x|^{-\nu})^k}{k!}.
		\]
		Setting \(l=q-k\) and dividing by \((q!)^{\mathfrak s}\), we use
		\(\binom{q-1}{k-1}\leq\binom qk=(q!)/(k!l!)\) and \(\binom qk^{\mathfrak s-2}\geq1\)
		(since \(\mathfrak s>2\)) to obtain
		\[
		\frac{|\partial_x^q e^{S(x)}|}{(q!)^{\mathfrak s}}
		\leq e^{S(x)}C^q
		\sum_{k+l=q}
		\frac{(C|x|^{-\nu-1})^k}{(k!)^{\mathfrak s}}
		\frac{(C|x|^{-1})^l}{(l!)^{\mathfrak s-1}} .
		\]
		Applying \eqref{eq7b} to both factors and evaluating at
		\(|x|\geq r\) yields
		\[
		|h^{-1}|_{\mathfrak s,K_0;\mathcal A(r)}
		\leq C\exp\bigl\{
		S(r)+Cr^{-(\nu+1)/\mathfrak s}+Cr^{-1/(\mathfrak s-1)}\bigr\}.
		\]
		Since \(1/(\mathfrak s-1)<\varkappa=(\nu+1)/\mathfrak s<\nu\), this
		gives \eqref{lem9c} for all \(K\geq K_0\).
	\end{proof}
	
	The correction windows are chosen inside each
	return base \(I_n=(-\rho_n,\rho_n)\). Set the explicit dyadic ratios
	\begin{equation*}
		(a_0,b_0,a_1,a_2,a_3):=\frac1{128}(12,13,14,15,16),
		\qquad a_*:=\frac{31}{256},
	\end{equation*}
	so that
	\begin{equation*}
		0<a_0<b_0<a_1<a_2<a_*<a_3=\frac18<\frac17,
		\qquad \frac23a_3=\frac1{12}<\frac{3}{32}=a_0.
	\end{equation*}
	At scale \(n\), we define the central core \(\mathcal K_n\), the central
	cutoff support \(\Omega_{n,0}\), the symmetric side plateaus
	\(P_{n,\pm}\), and the outer support \(\widetilde I_n\) by
	\begin{equation*}
		\begin{aligned}
			\mathcal K_n&=[-a_0\rho_n,a_0\rho_n],
			&\Omega_{n,0}&=(-b_0\rho_n,b_0\rho_n),\\
			P_{n,+}&=[a_1\rho_n,a_2\rho_n],
			&P_{n,-}&=-P_{n,+},\\
			\widetilde I_n&=(-a_3\rho_n,a_3\rho_n),
			&P_n&:=P_{n,-}\cup P_{n,+}.
		\end{aligned}
	\end{equation*}
	We have \(P_n\subset\mathcal A_n^0\) (since \(a_2<1/7\)) and
	\(|P_n|=p_*|I_n|\) with \(p_*:=a_2-a_1=1/128\). We also fix symmetric
	open neighborhoods \(\Omega_{n,\pm}\) of \(P_{n,\pm}\) obtained by
	rescaling a fixed interval inside \((b_0,a_*)\), so that
	\(P_{n,\pm}\Subset\Omega_{n,\pm}\Subset\widetilde I_n\setminus\overline{\Omega_{n,0}}\)
	and
	\begin{equation}\label{eq15}
		b_0\rho_n<|x|<a_*\rho_n,
		\qquad x\in\overline{\Omega_{n,-}\cup\Omega_{n,+}}.
	\end{equation}
	Because \(\rho_{n+1}\leq\frac23\rho_n\) and \(\frac23a_3<a_0\), the
	outer support at scale \(n+1\) lies compactly inside the central
	core at scale \(n\):
	\begin{equation}\label{eqsupp}
		\overline{\widetilde I_{n+1}}\Subset \mathcal K_n.
	\end{equation}
	On the positive half-interval \([0,\rho_n)\) (and symmetrically on \((-\rho_n,0]\)), these subintervals and partition points are ordered from the origin outward as
	\[
	0\in\widetilde I_{n+1}
	\Subset\mathcal K_n
	\Subset\Omega_{n,0},
	\qquad
	P_{n,+}\Subset\Omega_{n,+}\Subset\widetilde I_n\setminus\overline{\Omega_{n,0}}
	\Subset\mathcal A_n^0=(-d_{g_n},d_{g_n})\subset I_n,
	\]
	with relative endpoints satisfying
	\[
	0 < a_3\frac{\rho_{n+1}}{\rho_n} < a_0 < b_0 < a_1 < a_2 < a_* < a_3 = \frac18 < \frac17 \leq \frac{d_{g_n}}{\rho_n} < 1.
	\]
	Table~\ref{tab1} summarizes the geometry of these subregions inside \(I_n\) and the corresponding behavior of the scale-\(n\) cutoff functions and corrected connecting function \(g_n(E,x;v_n)\) constructed in Section~\ref{sec6}.
	
	\begin{table}[htbp]
		\centering
		\caption{Subregions of the return base \(I_n=(-\rho_n,\rho_n)\) and action of the scale-\(n\) perturbation \(\delta_n\).}
		\label{tab1}
		\footnotesize
		\setlength{\tabcolsep}{4.5pt}
		\renewcommand{\arraystretch}{1.25}
		\begin{tabular}{llll}
			\hline
			\textbf{Subregion} & \textbf{Range of \(|x|/\rho_n\)} & \textbf{Cutoffs} & \textbf{Behavior of \(g_n(E,x;v_n)\)} \\
			\hline
			Next support \(\widetilde I_{n+1}\) & \([0,\, a_3\rho_{n+1}/\rho_n)\) & \(\chi_{n,0}=1\) & Modified at scales \(m\geq n+1\) \\
			Central core \(\mathcal K_n\) & \([0,\, 12/128]\) & \(\chi_{n,0}=1\) & Reset to \(g_n(0,x;v_n)=-\lambda h(x)\) \\
			Central transition \(\Omega_{n,0}\setminus\mathcal K_n\) & \((12/128,\, 13/128)\) & \(0\leq\chi_{n,0}<1\) & Linear interpolation, \(g_n(0,x;v_n)<0\) \\
			Side buffers \(\Omega_{n,\pm}\setminus P_{n,\pm}\) & \(\subset(13/128,\, 31/256)\) & \(0\leq\chi_{n,\pm}<1\) & Logarithmic interpolant \(\mathcal H_{n,\pm}\) \\
			Side plateaus \(P_{n,\pm}\) & \([14/128,\, 15/128]\) & \(\chi_{n,\pm}=1\) & Exact resonance \(g_n(E_n,x;v_n)=0\) \\
			Outer collar \(I_n\setminus\widetilde I_n\) & \([16/128,\, 1)\) & \(\chi_{n,0}=\chi_{n,\pm}=0\) & Unchanged (\(\delta_n=0\)), \(g_n(0,x;v_n)<0\) \\
			\hline
		\end{tabular}
	\end{table}
	
	On \(\Omega_{n,0}\) and \(\Omega_{n,\pm}\), we choose \(G^\sigma\)
	cutoff functions \(\chi_{n,0}\) and \(\chi_{n,\pm}\) satisfying
	\[
	\{\chi_{n,0}=1\}=\mathcal K_n,
	\qquad \{\chi_{n,\pm}=1\}=P_{n,\pm}.
	\]
	We take \(0\leq\chi_{n,\omega}\leq1\), with each cutoff compactly
	supported in its corresponding \(\Omega_{n,\omega}\); all positive-order
	derivatives vanish at the support and plateau endpoints.
	The cutoffs are constructed from translated and rescaled copies of the
	standard \(G^\sigma\) step
	\[
	u(t)=
	\begin{cases}
		\exp(-t^{-1/(\sigma-1)}),&t>0,\\
		0,&t\leq0,
	\end{cases}
	\qquad \psi(t)=\frac{u(t)}{u(t)+u(1-t)}.
	\]
	On either side of each plateau, an increasing or decreasing copy of
	\(\psi\) joins the constant values \(0\) and \(1\). The inner endpoints
	of these transition intervals are the plateau endpoints, while their
	outer endpoints lie strictly inside the corresponding
	\(\Omega_{n,\omega}\), at fixed relative positions before rescaling.
	By \eqref{eq7a}
	and \(\mathfrak b=(\mathfrak s-\sigma)^{-1}\), there is a fixed
	\(K_*\geq K_0\) such that
	\begin{equation}\label{eqcut}
		|\chi_{n,\omega}|_{\mathfrak s,K_*}
		\leq C\exp\{C\rho_n^{-\mathfrak b}\},
		\qquad \omega\in\{0,-,+\},
	\end{equation}
	while on the side buffers \eqref{eq15},
	Lemma~\ref{lem9} gives
	\begin{equation}\label{eq16}
		|h|_{\mathfrak s,K_*;\Omega_{n,-}\cup\Omega_{n,+}}
		\leq C\exp\{-c\rho_n^{-\nu}\}.
	\end{equation}
	
	Finally, fix \(\vartheta_d\in(0,b_0/4)\) and define the scale-\(n\)
	resonant energy \(E_n\) and exponent \(\beta\) by
	\begin{equation}\label{eqEn}
		E_n:=\lambda h(\vartheta_d\rho_n),
		\qquad
		\beta:=\frac12\pi^{-\nu}\bigl(a_*^\nu+a_3^\nu\bigr)\vartheta_d^{-\nu}.
	\end{equation}
	The choice \(0<\vartheta_d<b_0/4<a_*<a_3\) ensures that \(\beta\)
	satisfies the strict inequalities
	\[
	\pi^{-\nu}<\pi^{-\nu}\Bigl(\frac{a_*}{\vartheta_d}\Bigr)^\nu
	<\beta<\pi^{-\nu}\Bigl(\frac{a_3}{\vartheta_d}\Bigr)^\nu.
	\]
	Using \(S(t)=|\sin\pi t|^{-\nu}=\pi^{-\nu}|t|^{-\nu}(1+O(t^2))\) as
	\(t\to0\), we may fix a starting index \(n_1\) large enough
	that \(\rho_{n_1}<\min\{\rho_0,1/6\}\) and, for all
	\(n\geq n_1\),
	\begin{equation}\label{eq17}
		E_n\geq\lambda e^{-\beta|x|^{-\nu}}
		\quad(x\in\Omega_{n,-}\cup\Omega_{n,+}),
		\qquad
		h(x)\geq e^{-\beta|x|^{-\nu}}
		\quad(0<|x|\leq\rho_{n_1}),
	\end{equation}
	while at the next scale,
	\begin{equation}\label{eq18}
		E_{n+1}\leq\frac1{100}
		\lambda\exp\bigl\{-\beta(a_3\rho_{n+1})^{-\nu}\bigr\},
		\qquad
		\frac{E_{n+1}}{E_n}
		\leq e^{-c\rho_n^{-\nu}}
		\leq\frac1{100}.
	\end{equation}
	Moreover, since \(\vartheta_d\rho_{n+1}\leq\frac23\vartheta_d\rho_n<\frac16b_0\rho_n<a_0\rho_n\)
	and \(\vartheta_d<b_0/4<1\), there is a constant \(c>0\) such that for
	all \(n\geq n_1\),
	\begin{gather}
		S(\vartheta_d\rho_{n+1})-S(a_0\rho_n)
		\geq c\rho_n^{-\nu},
		\label{eq19a}\\
		S(\vartheta_d\rho_n)-S(\rho_n)
		\geq c\rho_n^{-\nu}.
		\label{eq19b}
	\end{gather}

	\section{Localized potential perturbations and return-chain estimates}
	\label{sec6}
	
	\subsection{Affine gluing with logarithmic side interpolation}
	
	On a return base
	\(I\), the following localized affine perturbation prescribes both the
	zero-energy profile on the central core and the resonant zeros on the
	side plateaus, while preserving a strict sign throughout the cutoff
	transition zones.
	
	\begin{lemma}
		\label{lem10}
		Let \(I\subset\T\) be a return base and \(g_I(E,x;v)\) the normalized
		connecting function of \eqref{lem2b}. Let
		\(\Omega_0,\Omega_-,\Omega_+\Subset I\) be pairwise disjoint open
		subintervals, and let \(0\leq\chi_\omega\leq1\) (\(\omega\in\{0,-,+\}\))
		be smooth cutoffs in \(C_c^\infty(\Omega_\omega)\) that are flat to all
		orders at \(\partial\supp\chi_\omega\). Assume that the open \(I\)-gap
		blocks admit smooth SVD branches with nonvanishing cosine product on
		\([0,E_*]\times\bigcup_\omega\Omega_\omega\).
		
		Write \(g^-(E,x):=g_I(E,x;v)\), and suppose that on the central
		transition zone \(\{x\in\Omega_0:0<\chi_0(x)<1\}\),
		\begin{equation*}
			g^-(0,x)\leq-F(x)<0.
		\end{equation*}
		Let \(H_0\geq0\) and \(A>0\) be smooth on \(\Omega_0\) and
		\(\Omega_-\cup\Omega_+\), respectively, and let \(0<E_0\leq E_*\).
		Set
		\begin{equation*}
			e=g^-(0,\cdot)+A,
			\qquad
			B_\omega=\int_0^{E_0}\partial_Eg^-(u,\cdot)\,du,
			\qquad
			\mathcal H_\omega=A^{1-\chi_\omega}B_\omega^{\chi_\omega},
		\end{equation*}
		assuming \(B_\omega>0\) on \(\Omega_\omega\), and define the perturbed
		potential \(v^+=v+\delta\) by
		\begin{equation}\label{lem10a}
			\delta(x)=
			\frac{\chi_0(x)}\lambda\bigl(g^-(0,x)+H_0(x)\bigr)
			+\sum_{\omega\in\{-,+\}}
			\frac1\lambda\bigl(\mathcal H_\omega(x)-A(x)
			+\chi_\omega(x)e(x)\bigr),
		\end{equation}
		where each term is extended by zero outside its support. Then the open
		\(I\)-gap blocks are unchanged for all \(E\). On the stated domain, their
		SVD data are unchanged and \(g^+(E,x):=g_I(E,x;v^+)\) satisfies
		\begin{equation}\label{lem10b}
			g^+(E,x)=g^-(E,x)-\lambda\delta(x),
			\qquad \partial_E^qg^+(E,x)=\partial_E^qg^-(E,x)\quad(q\geq1).
		\end{equation}
		The following conclusions hold:
		\begin{enumerate}
			\item On the central plateau \(\{\chi_0=1\}\) and side plateaus
			\(\{\chi_\omega=1\}\) (\(\omega\in\{-,+\}\)),
			\begin{align}
				g^+(0,x)&=-H_0(x),&&\chi_0(x)=1,
				\label{lem10c}\\
				g^+(E_0,x)&=0,&&\chi_\omega(x)=1.
				\label{lem10d}
			\end{align}
			\item On the central transition zone \(\{0<\chi_0<1\}\),
			\begin{equation}\label{lem10e}
				g^+(0,x)\leq-\bigl[(1-\chi_0(x))F(x)+\chi_0(x)H_0(x)\bigr].
			\end{equation}
			\item On each side support \(\Omega_\omega\),
			\begin{equation}\label{lem10f}
				g^+(0,x)=-\mathcal H_\omega(x)+(1-\chi_\omega(x))e(x);
			\end{equation}
			in particular, if \(|(1-\chi_\omega)e|\leq\frac14\mathcal H_\omega\),
			then \(g^+(0,x)\leq-\frac34\mathcal H_\omega(x)\).
			\item If \(I'\Subset\Omega_0\) is disjoint from \(\Omega_-\cup\Omega_+\),
			then any further perturbation supported in \(I'\) preserves
			\eqref{lem10d} identically.
		\end{enumerate}
	\end{lemma}
	
	\begin{proof}
		Since \(\supp\delta\Subset I\), Lemma~\ref{lem3} shows
		that the open gap blocks and their SVD branches are unchanged, yielding
		\eqref{lem10b}. On \(\Omega_0\), we have
		\(g^+(0)=(1-\chi_0)g^-(0)-\chi_0H_0\), which gives both
		\eqref{lem10c} and
		\eqref{lem10e}. On \(\Omega_\omega\),
		substituting \(g^-(0)=-A+e\) into \(g^+(0)=g^-(0)-\lambda\delta\) gives
		\eqref{lem10f}. When \(\chi_\omega(x)=1\),
		\(\mathcal H_\omega(x)=B_\omega(x)\) and hence
		\(\lambda\delta(x)=B_\omega(x)-A(x)+e(x)=g^-(0,x)+B_\omega(x)=g^-(E_0,x)\),
		proving \eqref{lem10d}. Because \(\chi_\omega\) is
		flat to all orders at \(\partial\supp\chi_\omega\), the identity
		\(\mathcal H_\omega-A=A\bigl(e^{\chi_\omega\log(B_\omega/A)}-1\bigr)\)
		shows that \(\delta\) extends smoothly by zero across the support
		boundaries. Finally, any subsequent perturbation supported in
		\(I'\Subset\Omega_0\) vanishes on \(\Omega_\pm\) and at all intermediate
		orbit points of the open \(I\)-gaps, and hence preserves
		\eqref{lem10d}.
	\end{proof}
	
	\subsection{Gevrey bounds for logarithmic cutoff interpolants and reciprocals}
	
	By \eqref{eqcut} and the Banach algebra property of
	\(G^{\mathfrak s,K}\) for \(K\geq K_*\), the perturbation \(\delta\) in
	\eqref{lem10a} satisfies
	\begin{equation}\label{eq20}
		\begin{split}
			|\delta|_{\mathfrak s,K}
			\leq{}&\frac C\lambda e^{C\rho^{-\mathfrak b}}
			|g^-(0)+H_0|_{\mathfrak s,K;\Omega_0}\\
			&+\frac C\lambda\sum_{\omega\in\{-,+\}}
			\Bigl(
			|\mathcal H_\omega-A|_{\mathfrak s,K;\Omega_\omega}
			+e^{C\rho^{-\mathfrak b}}
			|e|_{\mathfrak s,K;\Omega_\omega}
			\Bigr).
		\end{split}
	\end{equation}
	The logarithmic interpolant \(\mathcal H_\omega\) and its reciprocal
	can be estimated without enlarging the Gevrey radius \(K\). The required
	composition estimate uses the endpoint flatness of the \(G^\sigma\)
	cutoffs.
	
	\begin{lemma}
		\label{lem11}
		Let \(1<\sigma<\mathfrak s\), set \(a=(\sigma-1)^{-1}\) and
		\(\mathfrak b=(\mathfrak s-\sigma)^{-1}\), and let \(\psi_\rho\) be the
		rescaling to an interval of length \(\asymp\rho\) of one of the fixed
		\(G^\sigma\) cutoff functions. Suppose that in the rescaled
		unit variable \(t\), a positive function \(\mathcal L_\rho\) satisfies
		\begin{equation}\label{lem11a}
			cL_0\leq\mathcal L_\rho(t)\leq CL_0,
			\qquad L_0=\rho^{-\nu},
			\qquad
			\|\partial_t^j\mathcal L_\rho\|_{C^0}
			\leq C^{j+1}j!L_0
			\quad(j\geq0).
		\end{equation}
		Assume that \(1/(\mathfrak s-1)<\nu\) and \(\mathfrak b<\nu\).
		Then, for \(\phi_\rho=\psi_\rho\) or \(\phi_\rho=1-\psi_\rho\),
		\begin{equation}\label{lem11b}
			\bigl|e^{-\phi_\rho\mathcal L_\rho}
			\bigr|_{\mathfrak s,K}
			\leq C e^{C\rho^{-\mathfrak b}},
			\qquad K\geq K_*.
		\end{equation}
		Moreover, if \(\Psi(w,z)\) is holomorphic on a fixed complex
		neighborhood of \([0,1]\times\{|z|\leq1/4\}\), then
		\begin{equation}\label{lem11c}
			|\Psi(\psi_\rho,\beta)|_{\mathfrak s,K}
			\leq C(\Psi) e^{C(\Psi)\rho^{-\mathfrak b}},
			\qquad |\beta|_{\mathfrak s,K}\leq\frac14,
			\quad K\geq K_*.
		\end{equation}
	\end{lemma}
	
	\begin{proof}
		Near an endpoint where \(\phi_\rho\) vanishes to infinite order, write
		\(\phi_\rho(t)=a_0(t)e^{-d(t)^{-a}}\), where \(d(t)\) is the distance to
		the endpoint in the unit variable \(t\). For the standard template
		\(u(t)=e^{-t^{-a}}\mathbf 1_{\{t>0\}}\) and
		\(\psi(t)=u(t)/(u(t)+u(1-t))\), the factor \(a_0(d)=(u(d)+u(1-d))^{-1}\)
		is holomorphic and bounded above and away from zero on a complex sector
		\(|z-d|\leq\epsilon d\). Set \(y=d^{-a}\), \(W=1+\log(e+L_0)\), and
		\(F(t)=e^{-\phi_\rho(t)\mathcal L_\rho(t)}\).
		
		Fix \(A>1\) large and \(\epsilon>0\) small. In the regime
		\(y\leq AW\), we apply Cauchy's estimate on the disc
		\(|z-t|\leq r_t:=\epsilon d/(1+y)\). On this disc,
		\(|(d+z-t)^{-a}-d^{-a}|\leq C(a)\epsilon\) and
		\(|\mathcal L_\rho(z)-\mathcal L_\rho(t)|\leq(c/10)L_0\), so
		\(\Re(\phi_\rho(z)\mathcal L_\rho(z))\geq0\) and \(|F(z)|\leq1\). Since
		\(r_t^{-1}=\epsilon^{-1}y^{\sigma-1}(1+y)\leq CW^\sigma\), Cauchy's
		inequality yields
		\[
		|F^{(r)}(t)|\leq r!(C/r_t)^r\leq C^{r+1}r!W^{\sigma r}.
		\]
		In the deep flat regime \(y>AW\), we use the larger disc
		\(|z-t|\leq\epsilon d\), where \(\Re(d+z-t)^{-a}\geq c(a)y\) and
		\(|\phi_\rho(z)\mathcal L_\rho(z)|\leq CL_0e^{-c(a)y}\leq1/2\). Applying
		Cauchy's estimate to \(F(z)-1\) and absorbing \(L_0\leq e^{c(a)y/2}\)
		gives, for \(r\geq1\),
		\[
		|F^{(r)}(t)|
		\leq C^{r+1}r!d^{-r}L_0e^{-c(a)y}
		\leq C^{r+1}r!\sup_{y>0}\bigl(y^{(\sigma-1)r}e^{-c(a)y/2}\bigr)
		\leq C^{r+1}(r!)^\sigma.
		\]
		Away from the flat endpoint, \(\phi_\rho\geq c>0\), so we may write
		\(F=H(-\log\phi_\rho-\log\mathcal L_\rho)\) with \(H(s)=e^{-e^{-s}}\).
		Since \(H\) has bounded derivatives on a horizontal strip and Fa\`a di
		Bruno's formula involves only positive-order derivatives of
		\(-\log\phi_\rho-\log\mathcal L_\rho\), we again obtain
		\(|F^{(r)}(t)|\leq C^{r+1}(r!)^\sigma\). Combining these estimates in
		the unit variable \(t\) gives
		\[
		\|\partial_t^rF\|_{C^0}
		\leq C^{r+1}\bigl((r!)^\sigma+r!W^{\sigma r}\bigr),
		\qquad r\geq0.
		\]
		Rescaling back to \(x=\rho t\), dividing by \(K^r(r!)^{\mathfrak s}\),
		and applying \eqref{eq7a} yields
		\[
		|F|_{\mathfrak s,K}
		\leq C\exp\bigl\{C\rho^{-1/(\mathfrak s-\sigma)}\bigr\}
		+C\exp\bigl\{C(\rho^{-1}W^\sigma)^{1/(\mathfrak s-1)}\bigr\}
		\leq Ce^{C\rho^{-\mathfrak b}},
		\]
		proving \eqref{lem11b}.
		
		For \eqref{lem11c}, expand
		\(\Psi(w,z)=\sum_{k\geq0}F_k(w)z^k\) on a relatively compact product
		neighborhood of \([0,1]\times\{|z|\leq R\}\) in the holomorphy domain,
		with \(R>1/4\). Let \(M<\infty\) be an upper bound for \(|\Psi|\) on this
		neighborhood. Applying the single-variable Fa\`a di Bruno formula to the
		rescaled cutoff \(\psi_\rho\)
		gives \(|F_k(\psi_\rho)|_{\mathfrak s,K}\leq CMR^{-k}e^{C\rho^{-\mathfrak b}}\),
		and the geometric series converges because \(|\beta|_{\mathfrak s,K}/R\leq(4R)^{-1}<1\).
		Summing it completes the proof.
	\end{proof}
	
	We now use Lemma~\ref{lem11} to
	estimate the scale-\(n\) correction \(\delta_n\) and the reciprocal
	\(g_n(E;v_n)^{-1}\) on the central transition zone and side buffers.
	
	\begin{lemma}
		\label{lem12}
		On a component \(\mathcal A\subset I_n\) where the return times are constant, let \(g_n(E,x;v_{n-1})\)
		be the normalized connecting function, and set \(A_n=\lambda h\) and
		\(e_n=g_n(0;v_{n-1})+A_n\). Fix \(K\geq K_*\), and assume that on
		the cutoff supports, uniformly for \(0\leq E\leq E_n\),
		\begin{equation}\label{lem12a}
			|e_n|_{\mathfrak s,K}\leq D,\qquad
			|\partial_Eg_n(E;v_{n-1})-1|_{\mathfrak s,K}
			\leq \delta_E\leq\frac18.
		\end{equation}
		On each side support \(\Omega_{n,\omega}\) (\(\omega\in\{-,+\}\)), define
		\begin{equation}\label{lem12b}
			\begin{split}
				B_{n,\omega}(x)
				&=\int_0^{E_n}\partial_Eg_n(u,x;v_{n-1})\,du
				=E_n\bigl(1+\beta_{n,\omega}(x)\bigr),\\
				\mathcal H_{n,\omega}
				&=A_n^{1-\chi_{n,\omega}}
				B_{n,\omega}^{\chi_{n,\omega}},
			\end{split}
		\end{equation}
		and define the scale-\(n\) potential perturbation by
		\begin{equation}\label{lem12c}
			\delta_n
			=\frac{\chi_{n,0}}\lambda e_n
			+\sum_{\omega\in\{-,+\}}\frac1\lambda
			\bigl(\mathcal H_{n,\omega}-A_n
			+\chi_{n,\omega}e_n\bigr),
			\qquad v_n=v_{n-1}+\delta_n.
		\end{equation}
		Then \(|\beta_{n,\omega}|_{\mathfrak s,K}\leq\delta_E\leq1/8\) and
		\(\frac78E_n\leq B_{n,\omega}\leq\frac98E_n\leq A_n\). The reference
		reciprocals and the side interpolation difference satisfy
		\begin{align}
			|A_n^{-1}|_{\mathfrak s,K;\Omega_{n,0}\setminus\mathcal K_n^\circ}
			&\leq C\lambda^{-1}
			\exp\bigl\{S(a_0\rho_n)+C\rho_n^{-\varkappa}\bigr\},
			\label{lem12d}\\
			|\mathcal H_{n,\omega}^{-1}|_{\mathfrak s,K;\Omega_{n,\omega}}
			&\leq CE_n^{-1}e^{C\rho_n^{-\mathfrak b}}
			\leq C\lambda^{-1}
			e^{C(\rho_n^{-\nu}+\rho_n^{-\mathfrak b})},
			\label{lem12e}\\
			|\mathcal H_{n,\omega}-A_n|_{\mathfrak s,K;\Omega_{n,\omega}}
			&\leq C\lambda
			e^{-c\rho_n^{-\nu}+C\rho_n^{-\mathfrak b}}.
			\label{lem12f}
		\end{align}
		Furthermore, there is a structural constant \(C'\geq1\) such that if
		\begin{equation}\label{lem12g}
			\lambda^{-1}\bigl(De^{C'\rho_n^{-\mathfrak b}}+E_{n+1}\bigr)
			e^{S(a_0\rho_n)+C'\rho_n^{-\varkappa}}\leq\frac18,
			\qquad
			E_n^{-1}(D+E_{n+1})e^{C'\rho_n^{-\mathfrak b}}\leq\frac18,
		\end{equation}
		then for all \(0\leq E\leq E_{n+1}\), the relative remainders on
		\(\Omega_{n,0}\setminus\mathcal K_n^\circ\) and \(\Omega_{n,\omega}\)
		have \(|\cdot|_{\mathfrak s,K}\)-norm at most \(1/8\), and
		\begin{align}
			|g_n(E;v_n)^{-1}|_{\mathfrak s,K}
			&\leq C\lambda^{-1}
			e^{C(\rho_n^{-\nu}+\rho_n^{-\mathfrak b})},
			\label{lem12h}\\
			|\partial_E(g_n(E;v_n)^{-1})|_{\mathfrak s,K}
			&\leq |g_n(E;v_n)^{-1}|_{\mathfrak s,K}^{2}
			|\partial_Eg_n(E;v_n)|_{\mathfrak s,K}.
			\label{lem12i}
		\end{align}
		Each norm here is taken on the central transition component or side
		component under consideration.
	\end{lemma}
	
	\begin{proof}
		Integrating \eqref{lem12a} over
		\([0,E_n]\) gives \(|\beta_{n,\omega}|_{\mathfrak s,K}\leq\delta_E\leq1/8\)
		and \(\frac78E_n\leq B_{n,\omega}\leq\frac98E_n\).
		On \(\Omega_{n,\omega}\), set \(\mathcal L_n(x):=\log(A_n(x)/E_n)=S(\vartheta_d\rho_n)-S(x)\)
		and subdivide \(\Omega_{n,\omega}\) at the plateau endpoints of
		\(P_{n,\omega}\) (across which \(\chi_{n,\omega}\) is flat to all orders).
		Because \(b_0\rho_n<|x|<a_*\rho_n\) and \(\vartheta_d<b_0/4\), in the
		rescaled coordinate \(t=x/\rho_n\) the holomorphic branch
		\((\sin(\pi\vartheta_d\rho_n))^{-\nu}-(\omega\sin(\pi\rho_n t))^{-\nu}\)
		satisfies \(c\rho_n^{-\nu}\leq\mathcal L_n\leq C\rho_n^{-\nu}\) (so in
		particular \(B_{n,\omega}\leq A_n\)) and Cauchy's estimate on a fixed
		complex annulus around \([b_0,a_*]\) gives
		\(\|\partial_t^j\mathcal L_n\|_{C^0}\leq C^{j+1}j!\rho_n^{-\nu}\),
		which verifies \eqref{lem11a}. The factorizations
		\[
		\mathcal H_{n,\omega}^{-1}
		=E_n^{-1}e^{-(1-\chi_{n,\omega})\mathcal L_n}
		(1+\beta_{n,\omega})^{-\chi_{n,\omega}},
		\qquad
		\mathcal H_{n,\omega}
		=A_ne^{-\chi_{n,\omega}\mathcal L_n}
		(1+\beta_{n,\omega})^{\chi_{n,\omega}},
		\]
		allow us to apply Lemma~\ref{lem11} with
		\(\Psi(w,z)=(1+z)^{\pm w}\). Together with \eqref{eq16},
		this proves \eqref{lem12e} and
		\eqref{lem12f}. The bound
		\eqref{lem9c} on \(|x|\geq a_0\rho_n\)
		gives \eqref{lem12d}.
		
		By Lemma~\ref{lem10}, we have
		\[
		g_n(E;v_n)=
		\begin{cases}
			-A_n+(1-\chi_{n,0})e_n+\int_0^E\partial_Eg_n(u;v_n)\,du
			&\text{on }\Omega_{n,0}\setminus\mathcal K_n^\circ,\\[1mm]
			-\mathcal H_{n,\omega}+(1-\chi_{n,\omega})e_n+\int_0^E\partial_Eg_n(u;v_n)\,du
			&\text{on }\Omega_{n,\omega}.
		\end{cases}
		\]
		Under the smallness condition \eqref{lem12g},
		multiplying the remainder terms by \(A_n^{-1}\) (respectively
		\(\mathcal H_{n,\omega}^{-1}\)) gives a relative remainder of
		\(G^{\mathfrak s,K}\)-norm at most \(1/8\). Inverting via a Neumann
		series yields \eqref{lem12h}, and the chain
		rule \(\partial_E(g_n^{-1})=-g_n^{-2}\partial_Eg_n\) gives
		\eqref{lem12i}.
	\end{proof}
	
	The scalar inverse bound and the connector formulas
	\eqref{eq10} give the following estimate
	for the normalized entries of any connecting matrix
	\(U=R_\Phi A_{v,E}(y)R_\Theta\).
	
	\begin{lemma}
		\label{lem13}
		Let \(\lambda\geq1\), \(Y\geq0\), and let
		\(U=R_\Phi A_{v,E}(y)R_\Theta=(u_{ij})\) be a connecting matrix on an
		interval \(I\) with \(g=E-\lambda v-\tan\Phi-\tan\Theta\).
		Assume that in \(\mathcal G_{\mathfrak s,K}^{(1)}\),
		\begin{equation*}
			|g^{-1}|_{\mathfrak s,K;1}\leq C\lambda^{-1}e^{CY},\qquad
			1+|\tan\Phi|_{\mathfrak s,K;1}
			+|\tan\Theta|_{\mathfrak s,K;1}\leq A,\qquad
			|\cos\Phi\cos\Theta|\geq\frac12.
		\end{equation*}
		Then there exists \(C'>0\), depending only on \(C,\mathfrak s,A\), such
		that
		\begin{equation}\label{lem13a}
			1+|u_{11}^{-1}|_{\mathfrak s,K;1}
			+\sum_{\xi\in\{u_{12},u_{21},u_{22}\}}
			|\xi/u_{11}|_{\mathfrak s,K;1}
			\leq e^{C'(Y+1)}.
		\end{equation}
	\end{lemma}
	
	\begin{proof}
		Since \(\cos\Phi\) and \(\cos\Theta\) have constant signs and modulus at
		least \(1/2\), Lemma~\ref{lem4} bounds
		\(\sec\Phi=\pm(1+\tan^2\Phi)^{1/2}\) and
		\(\sec\Theta=\pm(1+\tan^2\Theta)^{1/2}\) in
		\(\mathcal G_{\mathfrak s,K}^{(1)}\). Substituting these into
		\eqref{eq10} and using the Banach
		algebra property yields \eqref{lem13a}.
	\end{proof}
	
	\subsection{Hyperbolic concatenation of return blocks}
	
	When concatenating scale-\(n\) return blocks \(D_{\Lambda_i}\) separated
	by connectors \(U_i\), the loss \(Y_{a(i)}\) at \(U_i\) is absorbed by
	the shorter of its two adjacent gaps. The resulting Gevrey
	bounds are independent of the number of factors \(m\) in the
	return word.
	
	\begin{lemma}
		\label{lem14}
		Let \(\lambda\geq e\), \(0\leq\varepsilon\leq1/4\), and consider a
		return product on an interval \(I\) in SVD form:
		\begin{equation*}
			B=\varepsilon_B R_{\theta_{m+1}}
			\bigl(D_{\Lambda_{m+1}}U_mD_{\Lambda_m}\cdots
			U_1D_{\Lambda_1}\bigr)R_{\phi_1},
			\qquad \varepsilon_B\in\{\pm1\}.
		\end{equation*}
		Here each connector \(U_i\) belongs to \(\mathrm{SL}(2,\R)\).
		Suppose each gap block \(i\in\{1,\ldots,m+1\}\) has type
		\(T_i\in\{S,L\}\) and physical length \(\ell_i\geq R_{T_i}\) with
		\(R_{\min}:=\min\{R_S,R_L\}\) large, and each connector \(U_i\) is
		assigned to \(a(i)\in\{i,i+1\}\) satisfying
		\begin{equation*}
			\ell_{a(i)}=\min\{\ell_i,\ell_{i+1}\},
			\qquad
			\#\{i:a(i)=j\}\leq2.
		\end{equation*}
		Set \(Y_i:=\ell_i^\vartheta\) with \(0<\vartheta<1\), and \(r:=m+\sum_{i=1}^{m+1}\ell_i\).
		
		\begin{enumerate}
			\item Assume first that for \(0\leq E\leq E_*\), the factors satisfy the pointwise bounds
			\begin{equation*}
				\log\Lambda_i\geq(1-\varepsilon)\ell_i\log\lambda,
				\qquad \|U_i\|\leq M,
				\qquad |(U_i)_{11}|\geq\kappa_i,
			\end{equation*}
			where \(\log^+\kappa_i^{-1}\leq C(1+Y_{a(i)})\) and
			\begin{equation*}
				M^2\kappa_i^{-2}
				e^{-2(1-\varepsilon)\min(\ell_i,\ell_{i+1})\log\lambda}
				\leq 2^{-6}.
			\end{equation*}
			Then
			\begin{equation}\label{lem14a}
				\frac{\log\|B\|}{r\log\lambda}
				\geq1-\varepsilon
				-C\bigl(R_S^{\vartheta-1}+R_L^{\vartheta-1}
				+R_S^{-1}+R_L^{-1}\bigr).
			\end{equation}
			
			\item Assume further that all factors belong to
			\(\mathcal G_{\mathfrak s,K}^{(1)}\), that the endpoint tangents
			\(\tan\phi_1,\tan\theta_{m+1}\) have norm at most \(A_1\), and
			that for some \(\mu\in[1/2,2/3]\),
			\begin{align}
				|\Lambda_i^{-1}|_{\mathfrak s,K;1}
				&\leq Ce^{-\mu\ell_i\log\lambda+CY_i},
				\label{lem14b}\\
				1+|(U_i)_{11}^{-1}|_{\mathfrak s,K;1}
				+\sum_{\xi\in\{(U_i)_{12},(U_i)_{21},(U_i)_{22}\}}
				|\xi/(U_i)_{11}|_{\mathfrak s,K;1}
				&\leq Ce^{CY_{a(i)}}.
				\label{lem14c}
			\end{align}
			Then there exist \(c',C'>0\), depending only on the structural
			constants \(C,A_1,\mathfrak s\) and independent of \(m\), such
			that for \(R_{\min}\) sufficiently large,
			\begin{align}
				|\Theta(B)-\theta_{m+1}|_{\mathfrak s,K;1}
				&\leq C'e^{-c'\ell_{m+1}\log\lambda+C'\ell_{m+1}^\vartheta},
				\label{lem14d}\\
				|\Phi(B)-\phi_1|_{\mathfrak s,K;1}
				&\leq C'e^{-c'\ell_1\log\lambda+C'\ell_1^\vartheta},
				\label{lem14e}
			\end{align}
			and the reciprocal of the full product norm satisfies
			\begin{equation}\label{lem14f}
				|\|B\|^{-1}|_{\mathfrak s,K;1}
				\leq \exp\Bigl\{-\bigl(\mu-C'(R_S^{\vartheta-1}+R_L^{\vartheta-1}
				+R_S^{-1}+R_L^{-1})\bigr)r\log\lambda\Bigr\}.
			\end{equation}
		\end{enumerate}
	\end{lemma}
	
	\begin{proof}
		When \(m=0\), \(B=\varepsilon_B R_{\theta_1}D_{\Lambda_1}R_{\phi_1}\)
		and all assertions follow. Assume \(m\geq1\).
		
		To prove part~(1), set \(\Gamma=(1-\varepsilon)\log\lambda\) and
		\(A_i=M^2\kappa_i^{-2}e^{-2\Gamma\min(\ell_i,\ell_{i+1})}\leq 2^{-6}\).
		Applying the middle product \(P=D_{\Lambda_{m+1}}U_m\cdots U_1D_{\Lambda_1}\)
		to \(e_1\), we write the vector after \(D_{\Lambda_i}\) as
		\(\xi_i(1,s_i)^{\mathsf T}\) with \(s_1=0\). For
		\(U_i=\left(\begin{smallmatrix}a_i&b_i\\c_i&d_i\end{smallmatrix}\right)\),
		the projective slope satisfies the Riccati recursion
		\begin{equation*}
			s_{i+1}=\Lambda_{i+1}^{-2}
			\frac{c_i+d_i s_i}{a_i+b_i s_i},
			\qquad 1\leq i\leq m.
		\end{equation*}
		An induction using this recursion gives
		\(|s_i|\leq 2M\kappa_{i-1}^{-1}\Lambda_i^{-2}\leq 2A_{i-1}\leq 2^{-5}\)
		for \(2\leq i\leq m+1\). The denominator loss at the first connector
		is \(\eta_1=0\). For \(2\leq i\leq m\), the inequality
		\(\min(\ell_{i-1},\ell_i)+\min(\ell_i,\ell_{i+1})\leq 2\ell_i\)
		gives
		\[
		\eta_i:=\frac{M|s_i|}{\kappa_i}
		\leq \frac{2M^2}{\kappa_{i-1}\kappa_i}e^{-2\Gamma\ell_i}
		\leq 2(A_{i-1}A_i)^{1/2}\leq 2^{-5}<\frac14.
		\]
		Thus \(|\xi_{i+1}|\geq\Lambda_{i+1}|\xi_i||(U_i)_{11}|(1-\eta_i)\), and
		since \(\|B\|=\|P\|\geq|\xi_{m+1}|\) and \(\log(1-\eta_i)\geq-2\eta_i\),
		summing over \(i\) gives
		\[
		\log\|B\|
		\geq(1-\varepsilon)\sum_{i=1}^{m+1}\ell_i\log\lambda
		-C\sum_{i=1}^m\bigl(1+\ell_{a(i)}^\vartheta\bigr).
		\]
		Because each gap receives at most two assignments,
		\(\ell_j^\vartheta\leq R_{T_j}^{\vartheta-1}\ell_j\), and
		\(m\leq(R_S^{-1}+R_L^{-1})\sum_i\ell_i\), dividing by \(r\log\lambda\)
		yields \eqref{lem14a}.
		
		For part~(2), we concatenate the factors from the right. Define
		\(P_k=D_{\Lambda_k}U_{k-1}\cdots U_1D_{\Lambda_1}\) for
		\(1\leq k\leq m+1\), with SVD
		\(P_k=\varepsilon_k R_{\alpha_k}D_{M_k}R_{\beta_k}\), and define
		\(D_k=\sum_{j=1}^k\ell_j\) and \(\Xi_k=\sum_{j=1}^k(1+\ell_j^\vartheta)\).
		Let \(C\geq1\) be a common fixed constant in
		\eqref{lem14b}--\eqref{lem14c}.
		Choose \(C_*\geq1\), depending only on the structural constants,
		large enough to bound the normalized entries of the effective connectors
		when \(|bz|_{\mathfrak s,K;1}\leq1/4\) below, together with the two-block estimates of
		Lemma~\ref{lem5}. Set
		\[
		C_{**}:=C+2C_*+\log(CC_*)+1.
		\]
		These constants are fixed before \(R_{\min}\) is chosen. We prove
		simultaneously that
		\begin{align}
			|M_k^{-1}|_{\mathfrak s,K;1}
			&\leq\exp\bigl\{-\mu D_k\log\lambda+C_{**}\Xi_k\bigr\},
			\label{eq21a}\\
			|\tan\alpha_k|_{\mathfrak s,K;1}
			&\leq C_*\exp\bigl\{-\mu\ell_k\log\lambda+C_*\ell_k^\vartheta\bigr\}.
			\label{eq21b}
		\end{align}
		Write \(\delta_R=R_{\min}^{-1}+R_{\min}^{\vartheta-1}\). Increase
		\(R_{\min}\) so that
		\[
		(2C_{**}+C_*)\delta_R\leq\frac\mu2,
		\qquad 2C R_{\min}^{\vartheta-1}
		+\frac{2\log C}{R_{\min}}\leq\mu.
		\]
		All subsequent lower bounds on \(R_{\min}\) involve only these fixed
		constants and are uniform for \(\lambda\geq e\). For \(k=1\),
		\(M_1=\Lambda_1\) and \(\alpha_1=\beta_1=0\), so both estimates hold.
		At step \(k\),
		write \(U_k=u_{11}\left(\begin{smallmatrix}1&b\\c&d\end{smallmatrix}\right)\)
		and \(z=\tan\alpha_k\), so that \(P_{k+1}=\varepsilon_k D_{\Lambda_{k+1}}V_kD_{M_k}R_{\beta_k}\)
		with effective connector \(V_k=U_kR_{\alpha_k}\). Since
		\(\ell_{a(k)}\leq\ell_k\), \eqref{lem14c} and
		\eqref{eq21b} imply that for \(R_{\min}\)
		sufficiently large, \(|bz|_{\mathfrak s,K;1}+|cz|_{\mathfrak s,K;1}+|dz|_{\mathfrak s,K;1}\leq1/4\).
		Using the identities
		\[
		(V_k)_{11}^{-1}=\sqrt{1+z^2}\,u_{11}^{-1}(1+bz)^{-1},\quad
		\frac{(V_k)_{12}}{(V_k)_{11}}=\frac{b-z}{1+bz},\quad
		\frac{(V_k)_{21}}{(V_k)_{11}}=\frac{c+dz}{1+bz},\quad
		\frac{(V_k)_{22}}{(V_k)_{11}}=\frac{d-cz}{1+bz},
		\]
		we expand the inverse in a Neumann series and the square root in a
		binomial series to bound all four normalized quantities
		by \(C_*e^{C_*\ell_{a(k)}^\vartheta}\), with constants independent of
		\(k,m\). The constants in this bound depend only on
		\eqref{lem14c} and on the bound \(|bz|_{\mathfrak s,K;1}\leq1/4\),
		so they are covered by the choice of \(C_*\).
		Since \(\ell_{a(k)}\leq\ell_{k+1}\), \(\ell_{a(k)}\leq D_k\), and
		\(\Xi_k\leq\delta_R D_k\), the preceding choice of \(R_{\min}\) gives
		\[
		\mathcal H(V_k)
		\max\bigl\{|M_k^{-1}|_{\mathfrak s,K;1},
		|\Lambda_{k+1}^{-1}|_{\mathfrak s,K;1}\bigr\}^{2}
		\leq C_*C^2e^{-\mu R_{\min}}\leq2^{-6},
		\]
		after a further increase of \(R_{\min}\). Here \(\mathcal H(V_k)\)
		is the normalized connector quantity in
		\eqref{lem5a}. Thus the domination condition
		of Lemma~\ref{lem5} holds uniformly at every
		concatenation. That lemma yields
		\[
		|M_{k+1}^{-1}|_{\mathfrak s,K;1}
		\leq C_*|M_k^{-1}|_{\mathfrak s,K;1}|\Lambda_{k+1}^{-1}|_{\mathfrak s,K;1}e^{C_*\ell_{a(k)}^\vartheta},
		\qquad
		|\tan\alpha_{k+1}|_{\mathfrak s,K;1}
		\leq C_*e^{C_*\ell_{a(k)}^\vartheta}|\Lambda_{k+1}^{-1}|_{\mathfrak s,K;1}^2.
		\]
		Iterating the first inequality over the preceding concatenations gives
		\[
		\begin{aligned}
			|M_{k+1}^{-1}|_{\mathfrak s,K;1}
			&\leq C_*^k C^{k+1}
			\exp\Bigl\{-\mu D_{k+1}\log\lambda
			+C\sum_{j=1}^{k+1}\ell_j^\vartheta
			+C_*\sum_{i=1}^{k}\ell_{a(i)}^\vartheta\Bigr\}\\
			&\leq\exp\{-\mu D_{k+1}\log\lambda+C_{**}\Xi_{k+1}\}.
		\end{aligned}
		\]
		The second inequality follows because each gap is paired with at most two connectors and
		\(k\log C_*+(k+1)\log C\leq(k+1)\log(CC_*)\).
		For the slope estimate, \(\ell_{a(k)}\leq\ell_{k+1}\) and the
		second condition on \(R_{\min}\) allow us to absorb
		\(2C\ell_{k+1}^{\vartheta}+2\log C\) into
		\(\mu\ell_{k+1}\log\lambda\). This proves both
		\eqref{eq21a} and
		\eqref{eq21b} at \(k+1\), with the
		constants fixed above.
		
		At \(k=m+1\), \eqref{eq21b} gives the left
		endpoint bound \eqref{lem14d}. To obtain the
		full inverse bound from \eqref{eq21a}, put
		\(\Delta=R_S^{\vartheta-1}+R_L^{\vartheta-1}+R_S^{-1}+R_L^{-1}\).
		Then \(\Xi_{m+1}\leq\Delta\sum_i\ell_i\) and
		\(m\leq\Delta\sum_i\ell_i\). Hence the cumulative error and the
		\(m\) connector sites are both absorbed by \(C'\Delta r\log\lambda\),
		with \(C'\geq C_{**}+1\), proving
		\eqref{lem14f}.
		Finally, for the right endpoint angle \(\Phi(B)=\phi_1+\beta_{m+1}\),
		the right-slope bound \eqref{lem5e} of
		Lemma~\ref{lem5} gives
		\[
		|\beta_{k+1}-\beta_k|_{\mathfrak s,K;1}
		\leq C_*e^{C_*\ell_{a(k)}^\vartheta}|M_k^{-1}|_{\mathfrak s,K;1}^2.
		\]
		Since \(D_k-\ell_1\geq(k-1)R_{\min}\),
		\(\Xi_k-1-\ell_1^\vartheta\leq(R_{\min}^{-1}+R_{\min}^{\vartheta-1})(D_k-\ell_1)\),
		and \(\ell_{a(k)}^\vartheta\leq R_{\min}^{\vartheta-1}(D_k-\ell_1)\)
		for \(k\geq2\), the first condition on \(R_{\min}\) yields
		\[
		|\beta_{k+1}-\beta_k|_{\mathfrak s,K;1}
		\leq C_*e^{-2\mu\ell_1\log\lambda+(2C_{**}+C_*)(1+\ell_1^\vartheta)}\lambda^{-\mu(k-1)R_{\min}}.
		\]
		The sum of these increments over \(1\leq k\leq m\) is bounded by
		\[
		\frac{C_*}{1-\lambda^{-\mu R_{\min}}}
		e^{-2\mu\ell_1\log\lambda
			+(2C_{**}+C_*)(1+\ell_1^\vartheta)}.
		\]
		Increasing \(R_{\min}\) once more makes the denominator at least
		\(1/2\), uniformly for \(\lambda\geq e\). This proves
		\eqref{lem14e} with constants independent of
		the number of factors. Every series and product above is taken in
		\(\mathcal G_{\mathfrak s,K}^{(1)}\), so \(K\) remains unchanged.
	\end{proof}

	\section{The inductive step}
	\label{sec7}
	
	\subsection{Inductive setup and hypotheses}
	
	At scale \(n\), for a potential \(w\) on a regular \(I_n\)-return
	atom, we write the normalized connecting function in terms of the
	singular angles as
	\begin{equation*}
		g_n(E,x;w)=E-\lambda w(x)
		-\tan\Phi_n(E,x;w)-\tan\Theta_n(E,x;w).
	\end{equation*}
	Here \(\Phi_n(E,x;w)=\phi_{I_n}^+(E,x;w)\) and
	\(\Theta_n(E,x;w)=\theta_{I_n}^-(E,x;w)\) are the right SVD angle of the
	forward open gap block and the left SVD angle of the backward open gap
	block, respectively.
	We use two regions of \(I_n\): the outer region where
	\(\chi_{n,0}<1\), and the part of the central plateau lying outside
	\(I_{n+1}\):
	\begin{equation*}
		\mathcal O_n:=\{x\in I_n:\chi_{n,0}(x)<1\},\qquad
		\mathcal D_n:=\{x\in I_n\setminus I_{n+1}:
		\chi_{n,0}(x)=1\}.
	\end{equation*}
	The outer region \(\mathcal O_n\) contains
	\(\Omega_{n,0}\setminus\mathcal K_n\) and both side supports
	\(\Omega_{n,\pm}\). At the next scale,
	\(\mathcal O_{n+1}=(\mathcal O_{n+1}\cap\mathcal O_n)\mathbin{\dot\cup}(\mathcal O_{n+1}\cap\mathcal K_n)\).
	We denote by \(\mathscr C_n\) the finite family of open intervals
	obtained by subdividing the return atoms of \(I_n\) and \(I_{n+1}\) at
	the cutoff support and plateau boundaries at scales \(n,n+1\) and their
	finite pullbacks along the scale-\((n+1)\) return words. Across the cutoff boundaries, the return times to \(I_n\) and \(I_{n+1}\) remain constant, so the open gap blocks and their SVD branches extend smoothly,
	whereas return-partition boundaries remain excluded in
	\(\mathcal Z_{\rm ret}\). On each component \(\mathcal A\in\mathscr C_n\),
	we assign the gap index \(T_n(\mathcal A)=L\) if
	\(\mathcal A\subset\mathcal A_n^0\) and \(T_n(\mathcal A)=S\) if
	\(\mathcal A\subset I_n\setminus\overline{\mathcal A_n^0}\), matching the
	pairing in Lemma~\ref{lem7}.
	
	The constants in the induction are chosen in the following order:
	\begin{itemize}
		\item By Lemma~\ref{lem4}, on the domain \(|\Phi|_{\mathfrak s,K;1}+|\Theta|_{\mathfrak s,K;1}\leq1/2\)
		(where \(|\cos\Phi\cos\Theta|\geq1/2\)), there is a constant
		\(C_{\tan}\geq4\), depending only on \(\mathfrak s\), such that
		\(|\sec^2\Phi|_{\mathfrak s,K}\,\leq C_{\tan}/2\) and
		\(|\tan\Phi'-\tan\Phi|_{\mathfrak s,K;1}\leq C_{\tan}|\Phi'-\Phi|_{\mathfrak s,K;1}\).
		\item Fix \(B_0\geq1\) large enough to bound all scalar inverses
		obtained directly from
		Lemma~\ref{lem12} and
		\eqref{lem9c}, including the initial
		correction and the central reference on
		\(\mathcal O_{n+1}\cap\mathcal K_n\), where
		\(|x|\geq a_0\rho_{n+1}\). On this last region the type is \(L\);
		elsewhere it is the assigned type \(T\in\{S,L\}\). In each case the
		bound is
		\begin{equation*}
			\mathcal B_{n,T}:=B_0\lambda^{-1}e^{B_0Y_{n,T}}.
		\end{equation*}
		\item With \(B_0\) and \(C_{\tan}\) fixed, choose \(0<c_0\leq1\leq C_0\)
		using Lemmas~\ref{lem13} and
		\ref{lem14} to bound both the endpoint
		angle increments \(|\Phi_{n+1}-\Phi_n|_{\mathfrak s,K_\lambda;1}+|\Theta_{n+1}-\Theta_n|_{\mathfrak s,K_\lambda;1}\)
		and their tangent differences
		\(|(\tan\Phi_{n+1}-\tan\Phi_n)+(\tan\Theta_{n+1}-\tan\Theta_n)|_{\mathfrak s,K_\lambda;1}\)
		uniformly for \(\mu_n\in[1/2,2/3]\).
		\item The scale-\(n\) correction majorant (for \(n\geq n_0+1\)) is defined by
		\begin{equation*}
			\mathfrak d_n:=C\Bigl(e^{-c\rho_n^{-\nu}+C\rho_n^{-\mathfrak b}}
			+\lambda^{-1}e^{-c_0 R_{n-1,L}\log\lambda
				+C_0 R_{n-1,L}^{\vartheta}
				+C\rho_n^{-\mathfrak b}}\Bigr).
		\end{equation*}
		The initial correction \(\delta_{n_0}\) at scale \(n_0\) is estimated
		separately in Section~\ref{sec8}.
	\end{itemize}
	
	To propagate the hyperbolicity and reciprocal estimates across scales, the induction tracks five coupled properties: \textup{(H0)} bounds the global \(G^{\mathfrak s,K_\lambda}\) and \(C^0\) norms of the potential; \textup{(H1)} maintains uniform exponential growth of the open return-gap blocks; \textup{(H2)} keeps the singular angles near the coordinate axes and ensures uniform energy transversality \(\partial_E g_n\approx 1\); \textup{(H3)} preserves a negative upper bound for \(g_n(0,x;v_{n-1})\) outside the core while controlling its deviation from \(-\lambda h\) on \(\widetilde I_n\); and \textup{(H4)} bounds the Gevrey norm of \(g_n(E,x;v_{n-1})^{-1}\) on the outer region \(\mathcal O_n\).
	
	In \textup{(H3)}, \textup{(H4)}, and Conclusion~\textup{(C2)} below, \(\{\sigma_n\}_{n\geq n_0}\), \(\{\bar\sigma_n\}_{n\geq n_0}\), and \(\{p_n\}_{n\geq n_0}\) denote deterministic sequences of floor and Neumann-series coefficients defined recursively in the proof (see \eqref{eq22a}--\eqref{eq26}) that satisfy the uniform bounds
	\[
	\frac14\leq\sigma_{n+1}\leq\bar\sigma_n\leq\sigma_n\leq\frac12,
	\qquad
	1=p_{n_0}\leq p_n\leq p_{n+1}<2
	\qquad(n\geq n_0).
	\]
	
	\begin{proposition}
		\label{porp1}
		Let \(n\geq n_0+1\) and let \(v_{n-1}\) be the potential before the
		scale-\(n\) correction. Assume the following inductive hypotheses at
		scale \(n\):
		
		\begin{enumerate}[label=(H\arabic*),start=0]
			\item For a fixed Gevrey parameter \(K_\lambda\geq K_*\) and a
			finite bound \(M_\lambda\),
			\begin{equation}\label{porp1a}
				|v_{n-1}|_{\mathfrak s,K_\lambda;\T}
				+\sum_{r\geq n}\mathfrak d_r\leq M_\lambda,\qquad
				\|v_{n-1}\|_{C^0(\T)}+\sum_{r\geq n}\mathfrak d_r\leq 2.
			\end{equation}
			Consequently, for all \(0\leq E\leq E_{n_0}\) and \(\lambda\geq2\),
			\begin{equation}\label{porp1b}
				\sup_{x\in\T}
				\max\bigl\{\|A_{v_{n-1},E}(x)\|,
				\|A_{v_{n-1},E}(x)^{-1}\|\bigr\}
				\leq 4\lambda .
			\end{equation}
			
			\item Every regular open \(I_n\)-gap block \(G\) of type
			\(T\in\{S,L\}\) and length \(\ell\) has a simple top singular value for
			\(0\leq E\leq E_n\), and satisfies
			\begin{align*}
				\log\|G(E)\|&\geq(1-\varepsilon_n)\ell\log\lambda,
				\qquad 0\leq\varepsilon_n\leq\frac14,
				\\
				|\|G\|^{-1}|_{\mathfrak s,K_\lambda;1}
				&\leq e^{-\mu_n\ell\log\lambda},
				\qquad \frac12\leq \mu_n\leq\frac23.
			\end{align*}
			
			\item On \([0,E_n]\times\mathscr C_n\), the SVD angle lifts satisfy
			\begin{equation*}
				\max_{\mathcal A\in\mathscr C_n}
				\bigl(
				|\Phi_n|_{\mathfrak s,K_\lambda;1;\mathcal A}
				+|\Theta_n|_{\mathfrak s,K_\lambda;1;\mathcal A}
				\bigr)
				\leq\frac12,
				\qquad
				|\cos\Phi_n\cos\Theta_n|\geq\frac12,
			\end{equation*}
			and their energy derivatives satisfy
			\begin{equation*}
				\max_{\mathcal A\in\mathscr C_n}
				\bigl(
				|\partial_E\Phi_n|_{\mathfrak s,K_\lambda;\mathcal A}
				+|\partial_E\Theta_n|_{\mathfrak s,K_\lambda;\mathcal A}
				\bigr)\leq Z_n,
				\qquad C_{\tan}Z_n\leq\frac18.
			\end{equation*}
			
			\item On the outer region \((I_n\setminus\mathcal K_n^\circ)\cap\Omega_{\rm ret}\),
			\(g_n(0,x;v_{n-1})\) satisfies the signed floor
			\begin{equation*}
				g_n(0,x;v_{n-1})\leq-F_n^-(x)<0,
				\qquad
				F_n^-(x)\geq \sigma_n\lambda e^{-\beta|x|^{-\nu}},
			\end{equation*}
			while on \(\widetilde I_n\), the zero-energy profile defect satisfies
			\begin{equation}\label{porp1c}
				|g_n(0;v_{n-1})+\lambda h|_{\mathfrak s,K_\lambda;\widetilde I_n}
				\leq D_n:=C_0e^{-c_0R_{n-1,L}\log\lambda+C_0R_{n-1,L}^\vartheta}.
			\end{equation}
			
			\item On every component \(\mathcal A\in\mathscr C_n\) contained
			in \(\mathcal O_n\), the reciprocal \(g_n(E;v_{n-1})^{-1}\) satisfies
			\begin{equation}\label{porp1d}
				|g_n(E;v_{n-1})^{-1}|_{\mathfrak s,K_\lambda;\mathcal A}
				\leq p_n\mathcal B_{n,T_n(\mathcal A)},
				\qquad 0\leq E\leq E_{n+1}.
			\end{equation}
		\end{enumerate}
		
		Define \(\delta_n\) and \(v_n=v_{n-1}+\delta_n\) by
		\eqref{lem12b}--\eqref{lem12c}.
		Then for \(\lambda\) and \(n_0\) sufficiently large, the following
		conclusions hold:
		\begin{enumerate}[label=(C\arabic*)]
			\item On the central core and side plateaus, \(\delta_n\) prescribes the zero-energy profile and resonant zeros:
			\begin{equation*}
				g_n(0,x;v_n)=-\lambda h(x)\quad(x\in \mathcal K_n\cap\Omega_{\rm ret}),
				\qquad
				g_n(E_n,x;v_n)=0\quad(x\in P_n\cap\Omega_{\rm ret}).
			\end{equation*}
			
			\item For \(0\leq E\leq E_{n+1}\)
			and \(x\in (I_n\setminus\widetilde I_{n+1})\cap\Omega_{\rm ret}\), the corrected connecting function and connecting matrix entry satisfy
			\begin{equation}\label{porp1e}
				g_n(E,x;v_n)
				\leq-\bar\sigma_n\lambda e^{-\beta|x|^{-\nu}}<0,
				\qquad
				|(C_{I_n}(E,x;v_n))_{11}|
				\geq \frac12\bar\sigma_n\lambda e^{-\beta|x|^{-\nu}}.
			\end{equation}
			
			\item Every regular open
			\(I_{n+1}\)-gap block \(B\) of length \(r\) satisfies, for
			\(0\leq E\leq E_{n+1}\),
			\begin{equation*}
				\log\|B(E)\|\geq(1-\varepsilon_{n+1})r\log\lambda,
				\qquad
				|\|B\|^{-1}|_{\mathfrak s,K_\lambda;1}
				\leq e^{-\mu_{n+1}r\log\lambda},
			\end{equation*}
			where \(\varepsilon_{n+1}\leq\varepsilon_n+C\Delta_n^*\) and
			\(\mu_{n+1}:=\mu_n-C'\Delta_n^*\).
			
			\item On every component
			\(\mathcal A\in\mathscr C_n\) contained in \(I_{n+1}\), with gap index \(T=T_n(\mathcal A)\in\{S,L\}\), uniformly for \(0\leq E\leq E_{n+1}\),
			\begin{equation*}
				|\Phi_{n+1}-\Phi_n|_{\mathfrak s,K_\lambda;1;\mathcal A}
				+|\Theta_{n+1}-\Theta_n|_{\mathfrak s,K_\lambda;1;\mathcal A}
				\leq C_0e^{-c_0R_{n,T}\log\lambda+C_0R_{n,T}^\vartheta},
			\end{equation*}
			and the same bound holds with \(\Phi_{n+1}-\Phi_n\) and
			\(\Theta_{n+1}-\Theta_n\) replaced by \(\tan\Phi_{n+1}-\tan\Phi_n\) and
			\(\tan\Theta_{n+1}-\tan\Theta_n\).
			Here both scales are evaluated at \((E,x;v_n)\). Since \(\delta_n\)
			is supported in \(I_n\), the scale-\(n\) open gap blocks and their chosen
			singular angles agree with those for \(v_{n-1}\) by
			Lemma~\ref{lem3}.
			
			\item The energy derivative satisfies \(\partial_Eg_n(E;v_n)=\partial_Eg_n(E;v_{n-1})\), and at scale \(n+1\),
			\begin{equation*}
				Z_{n+1}\leq Z_n+C_0\sum_{T\in\{S,L\}}
				e^{-c_0R_{n,T}\log\lambda+C_0R_{n,T}^\vartheta}.
			\end{equation*}
			
			\item The perturbation satisfies \(|\delta_n|_{\mathfrak s,K_\lambda}\leq \mathfrak d_n\), with
			\(\sum_{n\geq n_0+1}\mathfrak d_n<\infty\), preserving \textup{(H0)} at
			scale \(n+1\).
			
			\item All subsequent corrections \(\delta_m\)
			(\(m>n\)) are supported in \(\widetilde I_m\Subset\mathcal K_n\), so that
			\(g_n(E_n,x;v_\infty)=0\) on \(P_n\cap\Omega_{\rm ret}\) and
			\eqref{porp1e} remains valid for \(v_\infty\).
			
			\item For \(0\leq E\leq E_{n+1}\), the reciprocal of the corrected connecting function satisfies
			\begin{equation}\label{porp1f}
				|g_n(E;v_n)^{-1}|_{\mathfrak s,K_\lambda;\mathcal A}
				\leq
				\begin{cases}
					p_n\mathcal B_{n,T_n(\mathcal A)},&\mathcal A\subset\mathcal O_n,\\
					\mathcal B_{n,L},&\mathcal A\subset\mathcal D_n,
				\end{cases}
			\end{equation}
			and in the dual-number norm on either region (with \(T=L\) on \(\mathcal D_n\)),
			\begin{equation}\label{porp1g}
				|g_n(E;v_n)^{-1}|_{\mathfrak s,K_\lambda;1;\mathcal A}
				\leq (2B_0+60B_0^2)\lambda^{-1}e^{2B_0Y_{n,T}}.
			\end{equation}
		\end{enumerate}
	\end{proposition}
	
	\begin{proof}
		Since \(A(q)=qE_{11}+R_{\pi/2}\) and \(A(q)^{-1}=qE_{22}+R_{-\pi/2}\)
		satisfy \(\|A(q)^{\pm1}\|\leq 1+|q|\), and \(E_{n_0}\leq\lambda/2\),
		\eqref{porp1a} implies
		\eqref{porp1b}.
		
		Before proceeding to the three main steps of the proof, we specify the recursive sequences \(\sigma_n,\bar\sigma_n\) and \(p_n\). Define the energy and angle-increment losses at scale \(n\) by
		\begin{align}
			\zeta_n^E
			&:=\frac{2E_{n+1}}{\lambda}
			\exp\bigl\{\beta(a_3\rho_{n+1})^{-\nu}\bigr\},
			\label{eq22a}\\
			\zeta_n^A
			&:=\frac {C_0}\lambda
			e^{-c_0R_{n,L}\log\lambda+C_0R_{n,L}^{\vartheta}
				+\beta(a_0\rho_{n+1})^{-\nu}}
			+\frac {C_0}\lambda\sum_{T\in\{S,L\}}
			e^{-c_0R_{n,T}\log\lambda+C_0R_{n,T}^{\vartheta}
				+\beta(\rho_n/7)^{-\nu}},
			\label{eq22b}
		\end{align}
		and put \(\zeta_n:=\zeta_n^E+\zeta_n^A\). By the choice of \(\beta\)
		in \eqref{eqEn}, we have
		\(\pi^{-\nu}\vartheta_d^{-\nu}-\beta a_3^{-\nu}>0\).
		Together with \(S(t)=\pi^{-\nu}|t|^{-\nu}(1+O(t^2))\), this gives
		\[
		\begin{aligned}
			\zeta_n^E
			&=2\exp\bigl\{-S(\vartheta_d\rho_{n+1})
			+\beta(a_3\rho_{n+1})^{-\nu}\bigr\}\\
			&\leq C\exp\{-c\rho_{n+1}^{-\nu}\}.
		\end{aligned}
		\]
		Since \(\rho_{n+1}\leq\frac23\rho_n\), the energy losses are summable.
		The angle losses are summable by \eqref{eq13}, so
		\(\sum_{n\geq n_0}\zeta_n\to0\) as \(n_0\to\infty\).
		Taking \(n_0\) large enough that \(\sum_{n\geq n_0}\zeta_n\leq1/4\) and setting
		\begin{equation}\label{eq23}
			\sigma_{n_0}:=\frac12,\qquad
			\bar\sigma_n:=\sigma_n-\zeta_n^E,\qquad
			\sigma_{n+1}:=\sigma_n-\zeta_n,
		\end{equation}
		we obtain
		\begin{equation}\label{eq24}
			\frac14\leq \sigma_{n+1}\leq\bar\sigma_n\leq \sigma_n\leq\frac12
			\qquad(n\geq n_0).
		\end{equation}
		Similarly, to track the Neumann-series amplification of scalar inverses on components of \(\mathcal O_{n+1}\cap\mathcal O_n\), we set
		\begin{equation}\label{eq25}
			t_n:=\frac{2B_0C_0}{\lambda}
			\sum_{T\in\{S,L\}}
			e^{-c_0R_{n,T}\log\lambda
				+(C_0+B_0)Y_{n,T}},
			\qquad
			p_{n_0}:=1,\qquad p_{n+1}:=\frac{p_n}{1-t_n}.
		\end{equation}
		For \(n_0\) sufficiently large, \(\sum_{n\geq n_0}t_n\leq1/8\), so that
		\begin{equation}\label{eq26}
			1\leq p_n\leq\exp\Bigl(2\sum_{j\geq n_0}t_j\Bigr)<2
			\qquad(n\geq n_0).
		\end{equation}
		
		\smallskip
		\noindent\emph{Step 1: The scale-\(n\) correction and signed floor.}
		By \textup{(H2)} and the definition of \(C_{\tan}\),
		\(\partial_Eg_n=1-\sec^2\Phi_n\,\partial_E\Phi_n-\sec^2\Theta_n\,\partial_E\Theta_n\)
		satisfies \(|\partial_Eg_n-1|_{\mathfrak s,K_\lambda}\leq C_{\tan}Z_n\leq 1/8\),
		so \(\frac78\leq\partial_Eg_n\leq\frac98\) pointwise and
		\(|\partial_Eg_n|_{\mathfrak s,K_\lambda}\leq\frac{4\pi^2}{3}+\frac18<15\).
		By \textup{(H3)}, \(|e_n|_{\mathfrak s,K_\lambda;\widetilde I_n}\leq D_n\).
		We have \(\rho_n^{-\nu}+\rho_n^{-\mathfrak b}=o(R_{n-1,L})\) by
		\eqref{eq13}. Thus both smallness conditions in
		\eqref{lem12g} hold with \(D=D_n\), using
		\eqref{eq18} and
		\eqref{eq19a}. Applying
		Lemmas~\ref{lem12} and
		\ref{lem10} with \(H_0=A=\lambda h\)
		and \(E_0=E_n\) gives \textup{(C1)}. Moreover, on \(\Omega_{n,\omega}\),
		\eqref{eq17} gives
		\(\min\{A_n,B_{n,\omega}\}\geq\frac78\lambda e^{-\beta|x|^{-\nu}}\), so
		\[
		g_n(0,x;v_n)\leq-\frac34\mathcal H_{n,\omega}(x)
		\leq-\frac{21}{32}\lambda e^{-\beta|x|^{-\nu}}
		\leq-\sigma_n\lambda e^{-\beta|x|^{-\nu}}.
		\]
		The same bound \(g_n(0,x;v_n)\leq-\sigma_n\lambda e^{-\beta|x|^{-\nu}}\)
		holds on \(\Omega_{n,0}\setminus\mathcal K_n^\circ\) by
		\eqref{lem10e}, on
		\(\mathcal K_n\setminus\widetilde I_{n+1}\) by \(g_n(0;v_n)=-\lambda h\),
		and outside the cutoff supports by \textup{(H3)}. Since every
		\(x\in I_n\setminus\widetilde I_{n+1}\) satisfies \(|x|\geq a_3\rho_{n+1}\),
		integrating \(\partial_Eg_n\leq 2\) over \([0,E]\) for \(0\leq E\leq E_{n+1}\)
		increases \(g_n\) by at most \(2E_{n+1}\leq\zeta_n^E\lambda e^{-\beta|x|^{-\nu}}\),
		which yields \textup{(C2)} since \(\sigma_n-\zeta_n^E=\bar\sigma_n\) and
		\(|\cos\Phi_n\cos\Theta_n|\geq1/2\).
		
		\smallskip
		\noindent\emph{Step 2: Reciprocal bounds for \(g_n(E;v_n)\) and size of \(\delta_n\).}
		On components of \(\mathcal O_n\) outside the cutoff supports, \(\delta_n=0\) and \eqref{porp1d} holds directly,
		while on the central transition zone and side supports
		Lemma~\ref{lem12} and
		\eqref{eq13} bound
		\(|g_n(E;v_n)^{-1}|_{\mathfrak s,K_\lambda}\) by \(\mathcal B_{n,T_n(\mathcal A)}\).
		On \(\mathcal A\subset\mathcal D_n\), we have \(\chi_{n,0}=1\) and
		\(|x|\geq\rho_{n+1}\), so
		\(g_n(E;v_n)=-\lambda h+\int_0^E\partial_Eg_n(u;v_n)\,du\). By
		\eqref{lem9c} and
		\eqref{eq19b},
		\begin{equation}\label{eq27}
			\Bigl|\Bigl(\int_0^E\partial_Eg_n(u;v_n)\,du\Bigr)
			(\lambda h)^{-1}\Bigr|_{\mathfrak s,K_\lambda;\mathcal A}
			\leq C E_{n+1}\lambda^{-1}
			e^{S(\rho_{n+1})+C\rho_{n+1}^{-\varkappa}}
			<\frac18.
		\end{equation}
		Inverting by a Neumann series and applying
		\eqref{eq13} yields
		\(|g_n(E;v_n)^{-1}|_{\mathfrak s,K_\lambda;\mathcal A}\leq\mathcal B_{n,L}\).
		Since \(p_n<2\) and \(|\partial_Eg_n|_{\mathfrak s,K_\lambda}<15\), the
		identity \(\partial_E(g_n^{-1})=-g_n^{-2}\partial_Eg_n\) gives
		\(|\partial_E(g_n^{-1})|_{\mathfrak s,K_\lambda}\leq 60\mathcal B_{n,T}^2\),
		which proves both \eqref{porp1f} and
		\eqref{porp1g} in \textup{(C8)}.
		
		Substituting \(|e_n|_{\mathfrak s,K_\lambda}\leq D_n\) and
		\eqref{lem12f} into
		\eqref{eq20} yields
		\(|\delta_n|_{\mathfrak s,K_\lambda}\leq\mathfrak d_n\). Since
		\(\mathfrak b<\nu\) and \(\rho_n^{-\mathfrak b}\leq CR_{n-1,L}^{\tau\mathfrak b}=o(R_{n-1,L})\)
		by \eqref{eq12} and
		\eqref{lem6b}, \(\sum_{n\geq n_0+1}\mathfrak d_n<\infty\),
		and the triangle inequality preserves \textup{(H0)}, proving \textup{(C6)}.
		
		\smallskip
		\noindent\emph{Step 3: Concatenation at scale \(n+1\).}
		Decompose any regular open \(I_{n+1}\)-gap block \(B\) using
		Lemma~\ref{lem7}. At each intermediate
		visit \(y\in I_n\setminus I_{n+1}\), the connector
		\(U_y=R_{\Phi_n(E,y)}A_{v_n,E}(y)R_{\Theta_n(E,y)}\) satisfies
		\(\|U_y\|\leq 4\lambda\), while \eqref{porp1e} and
		\eqref{eq13} give
		\(\log^+|(U_y)_{11}|^{-1}\leq CY_{n,a(y)}\). By \textup{(C8)} and
		Lemma~\ref{lem13}, the mixed Gevrey
		connector hypothesis \eqref{lem14c} of
		Lemma~\ref{lem14} holds with a constant
		depending only on \(B_0\) and \(\mathfrak s\). Since
		\(16\lambda^2|(U_y)_{11}|^{-2}e^{-2(1-\varepsilon_n)R_{n,T}\log\lambda}\to0\)
		uniformly, Lemma~\ref{lem14} yields
		\textup{(C3)}. For \textup{(C4)}, at any \(x\in\mathcal A\subset I_{n+1}\),
		the right endpoint angle \(\Phi_{n+1}(E,x)\) of \(B_{I_{n+1}}^+(E,x)\)
		is estimated using the first scale-\(n\) sub-gap of length \(\ell_1=r_{I_n}^+(x)-1\), while
		the left endpoint angle \(\Theta_{n+1}(E,x)\) of \(B_{I_{n+1}}^-(E,x)\)
		is estimated using the last scale-\(n\) sub-gap of length \(\ell_{m+1}=r_{I_n}^-(x)-1\). On
		\(\mathcal A\subset I_{n+1}\cap\mathcal A_n^0\) both endpoint sub-gaps
		have length \(R_{n,L}\), whereas on \(\mathcal A\subset I_{n+1}\setminus\overline{\mathcal A_n^0}\)
		both have length at least \(R_{n,S}\); hence
		\eqref{lem14d}--\eqref{lem14e}
		give \textup{(C4)} and \textup{(C5)}. Finally,
		\eqref{eqsupp} and
		Lemma~\ref{lem10} give \textup{(C7)}.
	\end{proof}

	\section{Initialization and convergence}
	\label{sec8}
	
	At the initial scale \(n_0(\lambda)\), the potential before correction
	is \(v_{n_0-1}=h\), whose one-step entries \(E-\lambda h(x)\) have
	uniformly large magnitude outside \(I_{n_0}\). A direct projective
	contraction argument establishes the initial return-block estimates,
	after which we complete the multi-scale induction.
	
	\subsection{Initial return blocks via projective contraction}
	
	\begin{lemma}
		\label{lem15}
		Let \(\mathscr A\) be a finite family of intervals with fixed return
		itineraries. On each \(\mathcal A\in\mathscr A\), consider a product of
		length \(\ell\geq1\):
		\[
		P(E,x)=A(a_\ell(E,x))\cdots A(a_1(E,x)),\qquad
		A(a)=\begin{pmatrix}a&-1\\1&0\end{pmatrix}.
		\]
		Suppose that, on every \(\mathcal A\in\mathscr A\), the entries satisfy
		\(|a_j(E,x)|\geq\mathcal D\) and \(\partial_Ea_j\equiv1\) for
		\(0\leq E\leq E_*\). Assume that the reciprocals \(a_j^{-1}\)
		have finite mixed Gevrey norm on their respective intervals for some
		common Gevrey parameter, uniformly for \(0\leq E\leq E_*\).
		Then there exists an absolute constant \(C_1\geq1\) such that, whenever
		\(\mathcal D>16C_1\), one can choose a finite Gevrey parameter
		\(K_1\geq K_*\) for which
		\begin{align}
			\log\|P(E,x)\|
			&\geq\ell\log(\mathcal D-2),
			\label{lem15a}\\
			|\|P\|^{-1}|_{\mathfrak s,K_1;1}
			&\leq C_1(C_1/\mathcal D)^\ell,
			\label{lem15b}\\
			|\Phi(P)|_{\mathfrak s,K_1;1}
			+|\Theta(P)|_{\mathfrak s,K_1;1}
			&\leq C_1\mathcal D^{-1},\\
			|\partial_E\Phi(P)|_{\mathfrak s,K_1}
			+|\partial_E\Theta(P)|_{\mathfrak s,K_1}
			&\leq C_1\mathcal D^{-2}.
			\label{lem15c}
		\end{align}
	\end{lemma}
	
	\begin{proof}
		Set \(c_1:=16\pi^2/3=4|1|_{\mathfrak s,K}\). The family
		\(\mathscr A\) has finitely many return words, each of finite length.
		The positive-order derivative terms in \eqref{def1a}
		contain factors \(K^{-r}\to0\) as \(K\to\infty\), so we may choose
		\(K_1\geq K_*\) large enough that
		\begin{equation*}
			\max_{\mathcal A,j}|a_j^{-1}|_{\mathfrak s,K_1}
			\leq\frac{c_1}{2\mathcal D},\qquad
			\max_{\mathcal A,j}|\partial_E(a_j^{-1})|_{\mathfrak s,K_1}
			\leq\frac{c_1}{2\mathcal D^2}.
		\end{equation*}
		Under the projective action
		\(A(a)\binom{u}{v}=(a-s)u\binom{1}{\mathcal T_a(s)}\) with \(s=v/u\) and
		\(\mathcal T_a(s)=(a-s)^{-1}=a^{-1}(1-sa^{-1})^{-1}\), the ball
		\(|s|_{\mathfrak s,K_1;1}\leq 2c_1/\mathcal D\) in
		\(\mathcal G_{\mathfrak s,K_1}^{(1)}\) is mapped into itself and
		\(\mathcal T_a\) contracts with factor \(4c_1^2\mathcal D^{-2}<1/4\).
		Iterating from \(s_0=0\) bounds the forward slope \(s=p_{21}/p_{11}\),
		while applying the same recursion to the transposed reversed product
		\(P^{\mathsf T}=JA(a_1)\cdots A(a_\ell)J\) with \(J=\diag(1,-1)\) bounds
		\(r=p_{12}/p_{11}\). Differentiating the slope recursion gives
		\(\partial_Es_j=-s_j^2(1-\partial_Es_{j-1})\); estimating this identity in
		\(G^{\mathfrak s,K_1}\) yields
		\(\sup_j|\partial_Es_j|_{\mathfrak s,K_1}\leq C\mathcal D^{-2}\).
		
		Factoring \(P=p_{11}\left(\begin{smallmatrix}1&r\\s&t\end{smallmatrix}\right)\)
		with \(p_{11}=\prod_{j=1}^\ell(a_j-s_{j-1})\) and \(t=rs+p_{11}^{-2}\)
		(since \(\det P=1\)), we obtain
		\(|p_{11}^{-1}|_{\mathfrak s,K_1;1}\leq(C_1/\mathcal D)^\ell\) and
		\(|r|_{\mathfrak s,K_1;1}+|s|_{\mathfrak s,K_1;1}+|t|_{\mathfrak s,K_1;1}\leq C\mathcal D^{-1}\),
		with \(|\partial_Er|+|\partial_Es|+|\partial_Et|\leq C\mathcal D^{-2}\).
		The pointwise bound \(|a_j-s_{j-1}|\geq\mathcal D-2\) gives
		\eqref{lem15a}, and applying analytic functional
		calculus (Lemma~\ref{lem4}) to the SVD of
		\(\left(\begin{smallmatrix}1&r\\s&t\end{smallmatrix}\right)\) yields
		\eqref{lem15b}--\eqref{lem15c}.
	\end{proof}
	
	\subsection{The multi-scale convergence theorem}
	
	\begin{theorem}
		\label{thm2}
		Let \(\alpha\) satisfy \eqref{eqDC} and let
		\(\mathfrak s>1+\tau\). Fix
		\begin{equation*}
			0<\xi<\frac1{\tau\nu},\qquad
			X_\lambda=(\log\lambda)^\xi,\qquad
			n_0(\lambda)=\min\{n\geq n_1:Q_n^+\geq X_\lambda\}.
		\end{equation*}
		Starting from \(v_{n_0-1}:=h\), define the initial correction
		\(\delta_{n_0}\) by
		\eqref{lem12b} and
		\eqref{lem12c}, and set
		\(v_{n_0}=h+\delta_{n_0}\). For \(n\geq n_0+1\), use the same
		formulas to define \(v_n=v_{n-1}+\delta_n\). Then, for all sufficiently
		large \(\lambda\), there exists \(K_\lambda\geq K_*\) such that
		\[
		v_\infty=h+\delta_{n_0}+\sum_{n\geq n_0+1}\delta_n
		\in G^{\mathfrak s,K_\lambda}(\T),
		\qquad \|v_\infty\|_{C^0(\T)}\leq2.
		\]
		The series converges in the norm of this Gevrey space.
		The parameter \(K_\lambda\) is chosen once, large enough to include the
		initial correction, and remains fixed thereafter. The hypotheses
		\textup{(H0)--(H4)} and conclusions \textup{(C1)--(C8)} of
		Proposition~\ref{porp1} hold at
		every scale \(n\geq n_0+1\). At \(n=n_0\), the initial gap
		blocks satisfy \textup{(H1),(H2)}, and the corrected potential satisfies
		\textup{(C1)--(C5),(C7),(C8)}. In particular,
		\(\mu_{n_0}=2/3\) and \(1/2\leq\mu_n\leq2/3\) for all \(n\geq n_0\).
		
		The asymptotic growth loss satisfies
		\begin{equation}\label{thm2a}
			\varepsilon_\infty(\lambda)
			:=\sup_{n\geq n_0(\lambda)}\varepsilon_n
			\leq\varepsilon_{n_0(\lambda)}
			+C\sum_{n\geq n_0(\lambda)}\Delta_n^*
			\longrightarrow0
			\qquad(\lambda\to\infty).
		\end{equation}
		For every \(y\in P_n\cap\Omega_{\rm ret}\) and \(n\geq n_0\), both
		adjacent \(I_n\)-return gaps have length \(\ell_n=Q_n^+-1\), and their
		singular values
		\(\Lambda_n^\pm(y)=\|B_{I_n}^\pm(E_n,y;v_\infty)\|\) satisfy
		\begin{equation}\label{thm2b}
			r_{I_n}^+(y)=r_{I_n}^-(y)=Q_n^+,\qquad
			\left|\frac{\log\Lambda_n^\pm(y)}{\ell_n\log\lambda}-1\right|
			\leq\tilde\varepsilon_\lambda:=
			\varepsilon_\infty(\lambda)+\frac{\log4}{\log\lambda}
			\longrightarrow0.
		\end{equation}
	\end{theorem}
	
	\begin{proof}
		\emph{Initial gap estimates.}
		The definition of \(n_0=n_0(\lambda)\) and
		Lemma~\ref{lem6} give, for large \(\lambda\),
		\(X_\lambda\leq Q_{n_0}^+\leq C(\gamma,\tau)X_\lambda^\tau\), and hence
		\begin{equation}\label{eq28}
			\rho_{n_0}^{-\nu}\leq C(\gamma,\tau)(Q_{n_0}^+)^\nu
			\leq C(\gamma,\tau)(\log\lambda)^{\xi\tau\nu}=o(\log\lambda),
			\qquad Q_{n_0}^-\geq c(\gamma,\tau)X_\lambda^{1/\tau}\longrightarrow\infty.
		\end{equation}
		For \(v=h\) and \(0\leq E\leq E_{n_0}\), every one-step matrix in an
		open \(I_{n_0}\)-gap is evaluated at \(|x|\geq\rho_{n_0}\), so
		\[
		|E-\lambda h(x)|\geq\mathcal D_{n_0}
		:=\lambda\bigl(h(\rho_{n_0})-h(\vartheta_d\rho_{n_0})\bigr)
		=\lambda^{1-o(1)}\geq\lambda^{3/4}.
		\]
		Applying Lemma~\ref{lem15} to
		the finite family of initial return words yields a common Gevrey
		parameter for which \textup{(H1),(H2)} hold with
		\begin{gather*}
			\mu_{n_0}=\frac23,\qquad
			\varepsilon_{n_0}\leq C\frac{1+\rho_{n_0}^{-\nu}}{\log\lambda}=o(1),
			\\
			|\Phi_{n_0}|_{\mathfrak s,K_\lambda;1}
			+|\Theta_{n_0}|_{\mathfrak s,K_\lambda;1}
			\leq C\mathcal D_{n_0}^{-1}\leq\frac14,
			\qquad Z_{n_0}\leq C\mathcal D_{n_0}^{-2}.
		\end{gather*}
		The defect before the initial correction is estimated separately:
		\begin{gather*}
			e_{n_0}:=g_{n_0}(0;h)+\lambda h
			=-\tan\Phi_{n_0}(0;h)-\tan\Theta_{n_0}(0;h),\\
			|e_{n_0}|_{\mathfrak s,K_\lambda;\widetilde I_{n_0}}
			\leq D_{n_0}:=C\mathcal D_{n_0}^{-1}.
		\end{gather*}
		The return-length bound \eqref{porp1c} will first
		be established at \(n_0+1\). On
		\(I_{n_0}\setminus\mathcal K_{n_0}^\circ\), we have
		\(|x|\geq a_0\rho_{n_0}\). Thus
		\eqref{eq28} and
		\eqref{eq17} imply
		\[
		D_{n_0}\leq\frac12\lambda e^{-\beta|x|^{-\nu}},\qquad
		g_{n_0}(0,x;h)\leq-\frac12\lambda e^{-\beta|x|^{-\nu}}.
		\]
		This proves the signed part of \textup{(H3)} with \(\sigma_{n_0}=1/2\).
		
		\emph{The initial correction.}
		The initial angle estimates give
		\(|\partial_Eg_{n_0}-1|_{\mathfrak s,K_\lambda}\leq1/8\) for large
		\(\lambda\). In the two smallness conditions
		\eqref{lem12g}, the terms involving
		\(D_{n_0}\) satisfy
		\begin{align*}
			&\lambda^{-1}D_{n_0}
			e^{S(a_0\rho_{n_0})+C\rho_{n_0}^{-\mathfrak b}
				+C\rho_{n_0}^{-\varkappa}}
			\\
			&\hspace{15mm}
			+E_{n_0}^{-1}D_{n_0}e^{C\rho_{n_0}^{-\mathfrak b}}
			\leq e^{-2\log\lambda+o(\log\lambda)}.
		\end{align*}
		The terms involving \(E_{n_0+1}\) tend to zero by
		\eqref{eq19a},
		\eqref{eq18}, and
		\(\mathfrak b,\varkappa<\nu\). Therefore
		Lemmas~\ref{lem12} and
		\ref{lem10} apply with \(D=D_{n_0}\).
		The core and plateau identities in \textup{(C1)} follow, while the signed
		transition estimates and the energy loss \(\zeta_{n_0}^E\) yield
		\textup{(C2)} as in Step~1 of
		Proposition~\ref{porp1}.
		
		The initial correction belongs to \(G^{\mathfrak s}\). Its disjoint
		supports and \(0<B_{n_0,\omega}\leq A_{n_0}\) imply
		\[
		\|\delta_{n_0}\|_{C^0}
		\leq\|h\|_{C^0}+C\lambda^{-1}\mathcal D_{n_0}^{-1}.
		\]
		We now enlarge \(K_\lambda\) once to include this correction; all
		preceding norm bounds remain valid. Since \(\sum_{n\geq n_0+1}\mathfrak d_n\to0\)
		as \(\lambda\to\infty\) and \(2\|h\|_{C^0}<1\), we obtain
		\[
		\|h+\delta_{n_0}\|_{C^0}
		+\sum_{n\geq n_0+1}\mathfrak d_n\leq2.
		\]
		Choosing
		\(M_\lambda>
		|h+\delta_{n_0}|_{\mathfrak s,K_\lambda}
		+\sum_{n\geq n_0+1}\mathfrak d_n\)
		establishes \textup{(H0)} at \(n_0+1\) and also yields the one-step bound
		\(4\lambda\) for the corrected initial potential.
		
		On any component of \(\mathcal O_{n_0}\) outside the cutoff supports,
		\[
		g_{n_0}(E;v_{n_0})=-\lambda h(1-\mathfrak r_{n_0}),\qquad
		\mathfrak r_{n_0}
		=\frac{E-\tan\Phi_{n_0}(E;h)-\tan\Theta_{n_0}(E;h)}{\lambda h}.
		\]
		For \(0\leq E\leq E_{n_0+1}\) and \(|x|\geq a_0\rho_{n_0}\),
		\eqref{lem9c} gives
		\[
		|\mathfrak r_{n_0}|_{\mathfrak s,K_\lambda}
		\leq C\lambda^{-1}(E_{n_0+1}+\mathcal D_{n_0}^{-1})
		e^{S(a_0\rho_{n_0})+C\rho_{n_0}^{-\varkappa}}<\frac18.
		\]
		The energy term is controlled by
		\eqref{eq19a}, and the remaining term is
		\(e^{-2\log\lambda+o(\log\lambda)}\). On \(\mathcal D_{n_0}\), the
		exact central correction and \(|x|\geq\rho_{n_0+1}\) give the Neumann
		bound \eqref{eq27}. These estimates, together
		with the corrected transition and side inverses, give \textup{(C8)}
		at \(n_0\) with \(p_{n_0}=1\). Its energy-derivative estimate follows
		from \(\partial_Eg^{-1}=-g^{-2}\partial_Eg\), as in Step~2 of the
		proposition.
		
		The initial estimates \textup{(H1),(H2)}, together with
		\textup{(C2),(C8)}, verify the connector and gap hypotheses of
		Lemma~\ref{lem14}. Repeating Step~3 of
		Proposition~\ref{porp1} therefore
		gives \textup{(C3)--(C5)} at \(n_0\), while support nesting gives
		\textup{(C7)}. Since
		\(\widetilde I_{n_0+1}\Subset\mathcal K_{n_0}\), the core identity
		and the endpoint estimate \textup{(C4)} for \(T=L\) yield the defect bound at scale \(n_0+1\):
		\[
		|g_{n_0+1}(0;v_{n_0})+\lambda h|_{
			\mathfrak s,K_\lambda;\widetilde I_{n_0+1}}
		\leq C_0e^{-c_0R_{n_0,L}\log\lambda+C_0R_{n_0,L}^{\vartheta}}
		=D_{n_0+1}.
		\]
		
		\emph{Propagation to the next scale.}
		We now verify that the inductive hypotheses pass from scale \(n\) to \(n+1\) for all \(n\geq n_0\). Choose \(\lambda\) sufficiently large that
		\[
		C'\sum_{r\geq n_0}\Delta_r^*\leq\frac16,\qquad
		C_0\sum_{r\geq n_0}\sum_{T\in\{S,L\}}
		e^{-c_0R_{r,T}\log\lambda+C_0R_{r,T}^{\vartheta}}\leq\frac14,
		\]
		that \(\varepsilon_{n_0}+C\sum_{r\geq n_0}\Delta_r^*\leq1/4\),
		and that
		\[
		C_{\tan}\left(Z_{n_0}+C_0\sum_{r\geq n_0}\sum_{T\in\{S,L\}}
		e^{-c_0R_{r,T}\log\lambda+C_0R_{r,T}^{\vartheta}}\right)\leq\frac18.
		\]
		The updates in \textup{(C3)--(C5)} then preserve
		\(\mu_n\in[1/2,2/3]\), \(\varepsilon_n\leq1/4\), and
		\textup{(H2)}, and yield \eqref{thm2a}.
		
		On \(\widetilde I_{n+1}\Subset\mathcal K_n\), both adjacent gaps
		from the preceding scale are long. The core identity and \textup{(C4)} give
		\eqref{porp1c} at \(n+1\). To prove the signed part,
		decompose
		\[
		I_{n+1}\setminus\mathcal K_{n+1}^\circ
		=(I_{n+1}\setminus\widetilde I_{n+1})\mathbin{\dot\cup}
		(\widetilde I_{n+1}\setminus\mathcal K_{n+1}^\circ).
		\]
		On the first set, \textup{(C2)} gives the floor with coefficient
		\(\bar\sigma_n\). The angle change is at most
		\(\zeta_n^A\lambda e^{-\beta|x|^{-\nu}}\): on the old central atom we use
		\(T=L\) and \(|x|\geq a_3\rho_{n+1}\), while on an old side atom we use
		\(T=S\) and \(|x|\geq\rho_n/7\). On the second set,
		\(g_n(0;v_n)=-\lambda h\), \(|x|\geq a_0\rho_{n+1}\), and \textup{(C4)} for \(T=L\) gives the same angle bound. Consequently the
		next signed floor holds with coefficient
		\(\sigma_{n+1}=\bar\sigma_n-\zeta_n^A\), proving \textup{(H3)}.
		
		For \textup{(H4)}, split
		\[
		\mathcal O_{n+1}
		=(\mathcal O_{n+1}\cap\mathcal O_n)\mathbin{\dot\cup}
		(\mathcal O_{n+1}\cap\mathcal K_n).
		\]
		On the first set, \textup{(C8)} gives
		\(|g_n(E;v_n)^{-1}|_{\mathfrak s,K_\lambda}\leq p_n\mathcal B_{n,T}\)
		with the scale-\(n\) gap index \(T\). On the second set,
		\(|x|\geq a_0\rho_{n+1}\), \(T=L\), and
		\[
		g_n(E;v_n)=-\lambda h+\int_0^E\partial_Eg_n(u;v_n)\,du.
		\]
		For \(0\leq E\leq E_{n+2}\),
		\eqref{lem9c} and
		\eqref{eq19a} at scale \(n+1\) give
		\[
		\left|\left(\int_0^E\partial_Eg_n(u;v_n)\,du\right)
		(\lambda h)^{-1}\right|_{\mathfrak s,K_\lambda}
		\leq CE_{n+2}\lambda^{-1}
		e^{S(a_0\rho_{n+1})+C\rho_{n+1}^{-\varkappa}}<\frac18.
		\]
		Neumann inversion and \eqref{eq13} therefore
		give \(|g_n(E;v_n)^{-1}|_{\mathfrak s,K_\lambda}\leq\mathcal B_{n,L}\)
		on this set, since \(B_0\) was chosen to cover the direct central
		reference bounds.
		
		On either set, define
		\(\Delta_n=(\tan\Phi_{n+1}-\tan\Phi_n)
		+(\tan\Theta_{n+1}-\tan\Theta_n)\).
		Then \(g_{n+1}(E;v_n)=g_n(E;v_n)-\Delta_n\), and
		\textup{(C4)} with the same gap index \(T\) gives
		\[
		|\Delta_ng_n(E;v_n)^{-1}|_{\mathfrak s,K_\lambda}\leq t_n<\frac18.
		\]
		Applying the Neumann series directly to
		\(f_n=g_n(E;v_n)^{-1}\) yields
		\[
		|g_{n+1}(E;v_n)^{-1}|_{\mathfrak s,K_\lambda}
		\leq\frac{p_n\mathcal B_{n,T}}{1-t_n}
		=p_{n+1}\mathcal B_{n,T}
		\leq p_{n+1}\mathcal B_{n+1,T'},
		\]
		because \(R_{n,L}\leq R_{n+1,S}\leq R_{n+1,L}\). At \(n=n_0\),
		this gives precisely the multiplier
		\(p_{n_0+1}=p_{n_0}/(1-t_{n_0})\) required by \textup{(H4)} at scale \(n_0+1\).
		
		All hypotheses now hold at \(n_0+1\). We can therefore apply
		Proposition~\ref{porp1} and the
		propagation argument successively to construct every subsequent scale.
		Conclusion \textup{(C6)} gives convergence at the fixed parameter
		\(K_\lambda\), and \textup{(C7)} preserves each resonant plateau.
		Finally, for \(y\in P_n\cap\Omega_{\rm ret}\), both adjacent gaps have
		length \(\ell_n=Q_n^+-1\), and
		\[
		(1-\varepsilon_\infty)\ell_n\log\lambda
		\leq\log\Lambda_n^\pm(y)\leq\ell_n\log(4\lambda).
		\]
		This proves \eqref{thm2b}.
	\end{proof}

	\subsection{Properties of the limiting potential}
	\label{sec9}
	
	Passing to the limit \(n\to\infty\) in
	Theorem~\ref{thm2} yields the following
	properties of \(v_\infty\in G^{\mathfrak s,K_\lambda}(\T)\) at every
	scale \(j\geq n_0(\lambda)\), where we set
	\begin{equation*}
		\Gamma_\lambda
		:=(1-\varepsilon_\infty(\lambda))\log\lambda>0.
	\end{equation*}
	
	\begin{corollary}
		\label{coro1}
		For all sufficiently large \(\lambda\) and every scale \(j\geq n_0(\lambda)\):
		\begin{enumerate}
			\item Every regular open \(I_j\)-gap block \(G\) for \(v_\infty\) of
			length \(|G|\geq Q_j^--1\) satisfies
			\begin{equation}\label{coro1a}
				\log\|G(E)\|\geq\Gamma_\lambda |G|,
				\qquad 0\leq E\leq E_j,
			\end{equation}
			and \(\sup_{x\in\T}\max\{\|A_{v_\infty,E}(x)\|,\|A_{v_\infty,E}(x)^{-1}\|\}\leq 4\lambda\)
			for \(0\leq E\leq E_{n_0}\).
			\item Pointwise on regular return atoms,
			\(|\cos\Phi_j(0,y)\cos\Theta_j(0,y)|\geq 1/2\) for
			\(y\in I_j\cap\Omega_{\rm ret}\), and on the annulus
			\((I_j\setminus\widetilde I_{j+1})\cap\Omega_{\rm ret}\),
			\begin{equation}\label{coro1b}
				|g_{I_j}(0,y;v_\infty)|
				\geq\frac\lambda4 e^{-\beta|y|^{-\nu}},
				\qquad
				|(C_{I_j}(0,y;v_\infty))_{11}|
				\geq\frac\lambda8 e^{-\beta|y|^{-\nu}}.
			\end{equation}
			Moreover, for every \(r>j\) and \(y\in I_r\cap\Omega_{\rm ret}\),
			\begin{equation}\label{coro1c}
				|g_{I_j}(0,y;v_\infty)-g_{I_r}(0,y;v_\infty)|
				\leq C_0\sum_{k=j}^{r-1}\sum_{T\in\{S,L\}}
				e^{-c_0R_{k,T}\log\lambda+C_0R_{k,T}^\vartheta}
				\leq Ce^{-c_0Q_j^-\log\lambda/2}.
			\end{equation}
			\item On the resonant plateaus \(P_j=P_{j,-}\cup P_{j,+}\), we have
			\(|P_j|=p_*|I_j|\) with \(p_*=1/128\), and for every
			\(y\in P_j\cap\Omega_{\rm ret}\),
			\begin{equation}\label{coro1d}
				r_{I_j}^+(y)=r_{I_j}^-(y)=Q_j^+,\qquad
				(C_{I_j}(E_j,y;v_\infty))_{11}=0,
			\end{equation}
			while the singular values \(\Lambda_{j,y}^\pm:=\|B_{I_j}^\pm(E_j,y;v_\infty)\|\)
			satisfy \eqref{thm2b}.
		\end{enumerate}
	\end{corollary}
	
	\begin{proof}
		Part~(1) follows from the initial estimate \textup{(H1)}, \textup{(C3)},
		the isolation of the open gaps from later corrections, and the uniform one-step bound. Part~(3) follows
		from \textup{(C1)}, \textup{(C7)}, and
		\eqref{thm2b}. In
		part~(2), \eqref{coro1b} follows from \textup{(C2)},
		\textup{(C7)}, and \(\bar\sigma_j\geq1/4\) in
		\eqref{eq24}. Finally, for \(j\leq k<r\)
		and \(y\in I_r\subset I_{k+1}\), the open \(I_k\)- and
		\(I_{k+1}\)-gap blocks avoid \(I_{k+1}\). Since
		\(\operatorname{supp}(v_\infty-v_k)\subset\widetilde I_{k+1}\), these
		blocks and their compatible singular directions are the same for
		\(v_k\) and \(v_\infty\). The explicit potential term cancels in
		\(g_{I_{k+1}}-g_{I_k}\). Applying \textup{(C4)} to the remaining tangent
		differences and summing over \(k=j,\ldots,r-1\) proves
		\eqref{coro1c}.
	\end{proof}

	\section{Lyapunov exponent bounds and proof of Theorem~\ref{thm1}}
	\label{sec10}
	
	The gap growth, connector floors, and resonant plateaus of
	Corollary~\ref{coro1} yield a lower bound for
	\(L_{v_\infty}(0)\) and upper bounds for \(L_{v_\infty}(E_j)\).
	
	\subsection{Lower bound at zero energy}
	
	Following the exceptional-set method of \cite[Section~7.1]{LTY2025}, fix
	exponents
	\begin{equation}\label{eq29}
		1<b<a<\frac1{\tau\nu},
	\end{equation}
	and for each scale \(j\geq n_0(\lambda)\), define the tiny central
	interval, orbit length, and exceptional phase set by
	\begin{equation*}
		I_j^*:=(-\rho_j^a,\rho_j^a),\qquad
		m_j:=\lfloor\rho_j^{-b}\rfloor,\qquad
		\mathfrak B_j:=\bigcup_{k=0}^{m_j}\bigl(I_j^*-k\alpha\bigr).
	\end{equation*}
	
	\begin{lemma}
		\label{lem16}
		As \(j\to\infty\), we have
		\(|\mathfrak B_j|\leq 4\rho_j^{a-b}\to0\),
		\(\rho_j^{-a\nu}\leq C(\gamma,\tau)(Q_j^-)^{\tau a\nu}=o(Q_j^-)\), and
		\(H_j/m_j\leq C\rho_j^{b-1}\to0\). Moreover, for all sufficiently large
		\(j\) and every \(x\in\Omega_{\rm ret}\setminus\mathfrak B_j\), every
		visit \(y=T^kx\in I_j\) (\(0\leq k\leq m_j\)) satisfies
		\begin{equation}\label{lem16a}
			|(C_{I_j}(0,y;v_\infty))_{11}|
			\geq\kappa_j:=\frac{\lambda}{32}
			\exp\bigl\{-\beta\rho_j^{-a\nu}\bigr\}.
		\end{equation}
	\end{lemma}
	
	\begin{proof}
		The measure bound \(|\mathfrak B_j|\leq 2(m_j+1)\rho_j^a\leq 4\rho_j^{a-b}\),
		the sublinear growth \(\rho_j^{-a\nu}\leq C(\gamma,\tau)(Q_j^-)^{\tau a\nu}=o(Q_j^-)\),
		and \(H_j/m_j\leq 2\rho_j^{-1}/m_j\leq C\rho_j^{b-1}\) follow directly
		from \eqref{lem6b}, \eqref{lem8a},
		and \eqref{eq29}.
		
		For \(x\in\Omega_{\rm ret}\setminus\mathfrak B_j\), any visit
		\(y=T^kx\in I_j\) with \(0\leq k\leq m_j\) lies in
		\((I_j\setminus I_j^*)\cap\Omega_{\rm ret}\), so \(|y|\geq\rho_j^a\).
		Let \(r:=\max\{k\geq j:y\in I_k\}\); then
		\(y\in I_r\setminus\widetilde I_{r+1}\), and
		\eqref{coro1b} at scale \(r\) gives
		\[
		|g_{I_r}(0,y;v_\infty)|
		\geq\frac{\lambda}{4}e^{-\beta|y|^{-\nu}}
		\geq\frac{\lambda}{4}e^{-\beta\rho_j^{-a\nu}}.
		\]
		Since \(\rho_j^{-a\nu}=o(Q_j^-)\), the telescoping bound
		\eqref{coro1c} implies that for all large \(j\),
		\[
		|g_{I_j}(0,y;v_\infty)-g_{I_r}(0,y;v_\infty)|
		\leq Ce^{-c_0Q_j^-\log\lambda/2}
		\leq\frac{\lambda}{8}e^{-\beta\rho_j^{-a\nu}}.
		\]
		Hence \(|g_{I_j}(0,y;v_\infty)|\geq\frac{\lambda}{8}e^{-\beta\rho_j^{-a\nu}}\),
		and multiplying by \(|\cos\Phi_j(0,y)\cos\Theta_j(0,y)|\geq 1/2\) yields
		\eqref{lem16a}.
	\end{proof}
	
	A uniform lower bound on the connecting entries gives the following
	growth estimate for a chain of hyperbolic blocks.
	
	\begin{lemma}\label{lem17}
		Let \(p\geq1\), \(R\geq1\), \(\Gamma>0\), \(M\geq1\), and
		\(0<\kappa\leq M\). Consider a matrix product
		\[
		W=D_{\Lambda_p}U_{p-1}D_{\Lambda_{p-1}}\cdots U_1D_{\Lambda_1},
		\]
		where each diagonal block has length \(\ell_i\geq R\) with
		\(\log\Lambda_i\geq\Gamma\ell_i\), and each connector satisfies
		\(\|U_i\|\leq M\) and \(|(U_i)_{11}|\geq\kappa\). If
		\[
		\eta:=16M^2\kappa^{-2}e^{-2\Gamma R}\leq\frac14,
		\]
		then for \(\ell=(p-1)+\sum_{i=1}^p\ell_i\),
		\begin{equation}\label{lem17a}
			\log\|W\|\geq
			\left(\Gamma-\frac{\Gamma+\log^+(\kappa^{-1})+2\eta}{R}
			\right)\ell.
		\end{equation}
	\end{lemma}
	
	\begin{proof}
		Write the vector obtained after \(D_{\Lambda_i}\) in \(We_1\) as
		\(x_i(1,s_i)^{\mathsf T}\) with \(s_1=0\). The projective recursion
		\(s_{i+1}=\Lambda_{i+1}^{-2}(u^{(i)}_{21}+u^{(i)}_{22}s_i)/(u^{(i)}_{11}+u^{(i)}_{12}s_i)\)
		and \(\Lambda_i^{-2}\leq e^{-2\Gamma R}\) then give inductively
		\(|s_i|\leq 3M\kappa^{-1}e^{-2\Gamma R}\leq\kappa/(2M)\leq 1/2\), so
		\(M|s_i|/\kappa\leq\eta\leq 1/4\) and
		\(|u^{(i)}_{11}+u^{(i)}_{12}s_i|\geq\kappa(1-\eta)\). Hence
		\begin{align*}
			\log\|W\|\geq\log|x_p|
			&\geq\sum_{i=1}^p\log\Lambda_i+(p-1)\log\kappa+(p-1)\log(1-\eta)\\
			&\geq\Gamma\bigl(\ell-(p-1)\bigr)-(p-1)\bigl(\log^+(\kappa^{-1})+2\eta\bigr),
		\end{align*}
		which yields \eqref{lem17a} since \((p-1)R\leq\sum_i\ell_i\leq\ell\).
	\end{proof}
	
	\begin{proposition}\label{porp2}
		For the limiting potential \(v_\infty\) of
		Theorem~\ref{thm2},
		\begin{equation*}
			L_{v_\infty}(0)
			\geq(1-\varepsilon_\infty(\lambda))\log\lambda.
		\end{equation*}
	\end{proposition}
	
	\begin{proof}
		For \(x\in\Omega_{\rm ret}\setminus\mathfrak B_j\), list all return times
		of the orbit segment \(x,Tx,\ldots,T^{m_j-1}x\) to \(I_j\) as
		\(0\leq t_0<t_1<\cdots<t_p<m_j\). By \eqref{lem8a}, the
		boundary fragments satisfy \(t_0\leq H_j\) and \(m_j-t_p\leq H_j+1\);
		since \(H_j/m_j\to0\), we have \(m_j>2H_j+1\) and hence \(p\geq1\) for
		all large \(j\). Write the open \(I_j\)-gap blocks as
		\(G_i=A_{v_\infty,0}^{(t_i-t_{i-1}-1)}(T^{t_{i-1}+1}x)=\epsilon_iR_{\theta_i}D_{\Lambda_i}R_{\phi_i}\)
		of length \(\ell_i=t_i-t_{i-1}-1\geq R_{j,S}=Q_j^--1\), so that
		\(\log\Lambda_i\geq\Gamma_\lambda\ell_i\) by
		\eqref{coro1a}. The interior product
		\(W_j(x)=A_{v_\infty,0}^{(t_p-t_0-1)}(T^{t_0+1}x)\) factors as
		\[
		W_j(x)=\Bigl(\prod_{i=1}^p\epsilon_i\Bigr)R_{\theta_p}
		\bigl(D_{\Lambda_p}U_{p-1}D_{\Lambda_{p-1}}\cdots U_1D_{\Lambda_1}\bigr)R_{\phi_1},
		\qquad
		U_i=C_{I_j}(0,T^{t_i}x;v_\infty).
		\]
		By Corollary~\ref{coro1} and
		Lemma~\ref{lem16}, \(\|U_i\|\leq 4\lambda\) and
		\(|(U_i)_{11}|\geq\kappa_j\), with
		\[
		\eta_j:=16(4\lambda)^2\kappa_j^{-2}e^{-2\Gamma_\lambda R_{j,S}}
		=2^{18}\exp\bigl\{2\beta\rho_j^{-a\nu}-2\Gamma_\lambda R_{j,S}\bigr\}
		\longrightarrow0\qquad(j\to\infty).
		\]
		The full cocycle iterate factors as
		\(A_{v_\infty,0}^{(m_j)}(x)=\mathcal U_j W_j(x)\mathcal V_j\), with
		boundary fragments \(\mathcal U_j=A_{v_\infty,0}^{(m_j-t_p)}(T^{t_p}x)\)
		and \(\mathcal V_j=A_{v_\infty,0}^{(t_0+1)}(x)\). Applying
		Lemma~\ref{lem17} to \(W_j(x)\) and using
		\(\log\|\mathcal U_j^{-1}\|+\log\|\mathcal V_j^{-1}\|\leq(2H_j+2)\log(4\lambda)\),
		we obtain, uniformly for \(x\in\Omega_{\rm ret}\setminus\mathfrak B_j\),
		\[
		\frac1{m_j}\log\|A_{v_\infty,0}^{(m_j)}(x)\|
		\geq\Gamma_\lambda
		-C\Bigl(\frac{\log\lambda}{Q_j^-}+(Q_j^-)^{\tau a\nu-1}+\rho_j^{b-1}\log\lambda\Bigr)
		=\Gamma_\lambda-o(1)\qquad(j\to\infty).
		\]
		Since \(\log\|A_{v_\infty,0}^{(m_j)}(x)\|\geq0\) for all \(x\in\T\) and
		\(|\mathfrak B_j|\to0\), integrating over \(\T\) and letting \(j\to\infty\)
		proves \(L_{v_\infty}(0)\geq\Gamma_\lambda=(1-\varepsilon_\infty(\lambda))\log\lambda\).
	\end{proof}
	
	\subsection{Upper bounds at the resonant energies}
	
	At each resonant energy \(E_j\) (\(j\geq n_0(\lambda)\)), every visit
	\(y\in P_j\cap\Omega_{\rm ret}\) lies in the central return atom
	\(\mathcal A_j^0\), so its forward and backward \(I_j\)-return gaps both
	have length \(Q_j^+-1\). The two-gap return block of total length
	\(2Q_j^+-1\) centered at \(y\),
	\begin{equation*}
		W_{j,y}(E_j)
		:=B_{I_j}^+(E_j,y;v_\infty)A_{v_\infty,E_j}(y)B_{I_j}^-(E_j,y;v_\infty),
	\end{equation*}
	coincides factor by factor with the stage-\(j\) block by
	\eqref{eqsupp} and admits the SVD factorization
	\begin{equation}\label{eq30}
		W_{j,y}(E_j)
		=\epsilon_{j,y}^+\epsilon_{j,y}^-R_{\theta_{I_j}^+(E_j,y;v_\infty)}D_{\Lambda_{j,y}^+}
		C_{I_j}(E_j,y;v_\infty)D_{\Lambda_{j,y}^-}
		R_{\phi_{I_j}^-(E_j,y;v_\infty)},
	\end{equation}
	where \((C_{I_j}(E_j,y;v_\infty))_{11}=0\) by
	\eqref{coro1d} and \(\Lambda_{j,y}^\pm\) satisfy the
	balance estimate \eqref{thm2b}.
	
	\begin{lemma}\label{lem18}
		For each fixed \(j\geq n_0(\lambda)\) and every sufficiently large
		\(m\), the orbit segment \(x,Tx,\ldots,T^{m-1}x\) of any
		\(x\in\Omega_{\rm ret}\) contains a collection of pairwise disjoint
		centered resonant blocks \(W_{j,y_r}(E_j)\) (\(y_r\in P_j\cap\Omega_{\rm ret}\))
		whose total length \(\mathcal L_{m,j}(x)\) satisfies
		\begin{equation}\label{lem18a}
			\mathcal L_{m,j}(x)\geq\frac{p_*}{8}m.
		\end{equation}
	\end{lemma}
	
	\begin{proof}
		By unique ergodicity of the irrational rotation \(T\), for all
		sufficiently large \(m\) and all \(x\in\T\),
		\[
		\#\{0\leq k<m:T^kx\in P_j\}\geq \frac12m|P_j|.
		\]
		Since \(r_{I_j}^+(y)=r_{I_j}^-(y)=Q_j^+\) on \(P_j\), consecutive visits
		to \(P_j\) are separated by at least \(Q_j^+\) steps. We retain every
		second visit to \(P_j\) and discard at most the first and last visits
		if their \((2Q_j^+-1)\)-blocks extend beyond \([0,m-1]\). This leaves at least
		\(\frac18m|P_j|\) pairwise disjoint centered blocks of length
		\(2Q_j^+-1\geq Q_j^+\) each. By \eqref{lem8b}, their
		total length is at least
		\[
		\mathcal L_{m,j}(x)
		\geq\frac18m|P_j|Q_j^+
		=\frac{p_*}{8}m|I_j|Q_j^+
		\geq\frac{p_*}{8}m.
		\]
	\end{proof}
	
	\begin{proposition}\label{porp3}
		For all sufficiently large \(\lambda\) and all \(j\geq n_0(\lambda)\),
		\begin{equation}\label{porp3a}
			L_{v_\infty}(E_j)
			\leq(1-c_*)\log\lambda,
			\qquad c_*:=\frac{p_*}{64}=\frac1{8192}.
		\end{equation}
	\end{proposition}
	
	\begin{proof}
		Fix \(j\geq n_0(\lambda)\) and \(y\in P_j\cap\Omega_{\rm ret}\), and set
		\(q=Q_j^+\) and \(R=2q-1\). Since
		\(C=C_{I_j}(E_j,y;v_\infty)=\left(\begin{smallmatrix}0&b\\c&d\end{smallmatrix}\right)\)
		has vanishing \((1,1)\)-entry and \(\|C\|=\|A_{v_\infty,E_j}(y)\|\leq 4\lambda\),
		the product \(D_{\Lambda_{j,y}^+}CD_{\Lambda_{j,y}^-}\) equals
		\[
		D_{\Lambda_{j,y}^+}CD_{\Lambda_{j,y}^-}
		=\begin{pmatrix}
			0&b\Lambda_{j,y}^+/\Lambda_{j,y}^-\\
			c\Lambda_{j,y}^-/\Lambda_{j,y}^+&d/(\Lambda_{j,y}^+\Lambda_{j,y}^-)
		\end{pmatrix}.
		\]
		Consequently, \eqref{eq30} and the balance bound
		\eqref{thm2b} give
		\[
		\log\|W_{j,y}(E_j)\|
		\leq|\log\Lambda_{j,y}^+-\log\Lambda_{j,y}^-|+\log(12\lambda)
		\leq 2\tilde\varepsilon_\lambda(q-1)\log\lambda+\log\lambda+\log 12.
		\]
		Since \(2(q-1)/R\geq 2/3\) for \(q\geq2\) and \(\tilde\varepsilon_\lambda\to0\)
		as \(\lambda\to\infty\), for all sufficiently large \(\lambda\) every
		resonant block satisfies the cancellation bound
		\begin{equation}\label{eq31}
			\log\|W_{j,y}(E_j)\|
			\leq\frac34 R\log\lambda.
		\end{equation}
		
		Enumerate the disjoint resonant blocks of Lemma~\ref{lem18}
		chronologically as \(\mathcal W_1,\ldots,\mathcal W_N\). The full iterate then factors as
		\(A_{v_\infty,E_j}^{(m)}(x)=F_N\mathcal W_NF_{N-1}\cdots F_1\mathcal W_1F_0\),
		where the intervening products \(F_r\) have total length
		\(m-\mathcal L_{m,j}(x)\) and each remaining one-step factor has norm at
		most \(4\lambda\). Submultiplicativity together with
		\eqref{lem18a} and \eqref{eq31} yields, for all
		large \(m\) and \(x\in\Omega_{\rm ret}\),
		\[
		\frac1m\log\|A_{v_\infty,E_j}^{(m)}(x)\|
		\leq\left(1-\frac{p_*}{32}\right)\log\lambda+\log 4
		\leq\left(1-\frac{p_*}{64}\right)\log\lambda
		\]
		once \(\log 4\leq(p_*/64)\log\lambda\). Integrating over \(\T\) and
		taking \(m\to\infty\) proves \eqref{porp3a}.
	\end{proof}
	
	\subsection{Proof of the main theorem}
	\label{sec11}
	
	\begin{proof}[Proof of Theorem~\ref{thm1}]
		By Theorem~\ref{thm2}, for every
		sufficiently large \(\lambda\) the limiting potential
		\(v_\infty\in G^{\mathfrak s}(\T,\R)\) is well defined and the resonant
		energies \(E_j=\lambda h(\vartheta_d\rho_j)\) (\(j\geq n_0(\lambda)\))
		satisfy \(E_j\downarrow0\). Propositions~\ref{porp2}
		and~\ref{porp3} yield
		\[
		L_{v_\infty}(0)\geq(1-\varepsilon_\infty(\lambda))\log\lambda,
		\qquad
		\sup_{j\geq n_0(\lambda)}L_{v_\infty}(E_j)\leq(1-c_*)\log\lambda,
		\]
		with \(c_*=1/8192\) and \(\varepsilon_\infty(\lambda)\to0\) as
		\(\lambda\to\infty\). Taking \(\lambda\) large enough that
		\(\varepsilon_\infty(\lambda)\leq c_*/2\) and reindexing the sequence
		\(\{E_j\}_{j\geq n_0(\lambda)}\) completes the proof of Theorem~\ref{thm1}.
	\end{proof}
	
	\section*{Acknowledgements}
	
	The authors would like to appreciate Professor Y. Wang and S. Jitomirskaya for their helpful discussions.
	
	This work was supported by the National Natural Science Foundation
	of China (Grant Nos.~12271245 and~12671221).
	
	The main mathematical ideas and the overall proof strategy were developed
	by the authors. During revision of this manuscript, AI tools were used to
	assist with checking calculations and polishing the text. The authors are
	responsible for the mathematical content, references, and final text.

\end{document}